\documentclass[11pt,reqno]{amsart}
\usepackage[a4paper]{geometry}
\usepackage{amsthm,hyperref,lmodern,mathrsfs}
\usepackage[latin1]{inputenc}
\usepackage[T1]{fontenc} 
\usepackage{graphicx}
\usepackage[frenchb,english]{babel}
\usepackage{calc}
\usepackage{amsmath,amsthm}

\usepackage{url}
\usepackage{times, amssymb, amscd, mathrsfs, graphicx, color}  
\usepackage{enumerate}

\newtheorem{theo}{Theorem}[section]

\newtheorem{defi}[theo]{Definition}
\newtheorem{coro}[theo]{Corollary}
\newtheorem{lem}[theo]{Lemma}

\newtheorem{hypo}[theo]{Assumption}
\newtheorem{Rq}[theo]{Remark}

\def\restriction#1#2{\mathchoice
              {\setbox1\hbox{${\displaystyle #1}_{\scriptstyle #2}$}
              \restrictionaux{#1}{#2}}
              {\setbox1\hbox{${\textstyle #1}_{\scriptstyle #2}$}
              \restrictionaux{#1}{#2}}
              {\setbox1\hbox{${\scriptstyle #1}_{\scriptscriptstyle #2}$}
              \restrictionaux{#1}{#2}}
              {\setbox1\hbox{${\scriptscriptstyle #1}_{\scriptscriptstyle #2}$}
              \restrictionaux{#1}{#2}}}
\def\restrictionaux#1#2{{#1\,\smash{\vrule height .8\ht1 depth .85\dp1}}_{\,#2}}

\newcommand{\1}{\mathbf{1}} 
\newcommand{\N}{\mathbb{N}}                                              % diverses redfinitions :
\newcommand{\R}{\mathbb{R}}                                              % \mathbb{R}, mais simplement \R.

\newcommand{\p}{\mathbb{P}}  
\newcommand{\E}{\mathbb{E}}                                            %pr avoir les P de proba tt de suite et E pr l esperance
\newcommand{\egalloi}{\stackrel{d}{=}}

\newcommand{\nocontentsline}[3]{}
\newcommand{\tocless}[2]{\bgroup\let\addcontentsline=\nocontentsline#1{#2}\egroup}

\title{Limit theorems for some branching measure-valued processes}
\author{Bertrand \textsc{Cloez}}
\address[B.~Cloez]{Laboratoire d'Analyse et de Math\'ematiques Appliqu\'ees, CNRS
  UMR8050, Universit\'e Paris-Est Marne-la-Vall\'ee, France} %
\email{\url{mailto:bertrand.cloez(at)univ-mlv.fr}} %
\urladdr{\url{http://bertrand.cloez.perso.sfr.fr/}}

\date{ Compiled \today}

\begin{document}
\maketitle

\begin{abstract}

We consider a particle system in continuous time, discrete population, with spatial motion and nonlocal branching. The offspring's weights and their number may depend on the mother's weight. Our setting captures, for instance, the processes indexed by a Galton-Watson tree. Using a size-biased auxiliary process for the empirical measure, we determine this asymptotic behaviour. We also obtain a large population approximation as weak solution of a growth-fragmentation equation. Several examples illustrate our results.

\end{abstract}

{\footnotesize %
%\noindent\textbf{Keywords.} Network Protocols; Queueing Theory; Additive
%Increase Multiplicative Decrease Processes (AIMD); Piecewise Deterministic
%Markov Processes (PDMP); Exponential Ergodicity; Coupling.

\medskip

%\noindent\textbf{AMS-MSC.}  68M12 ; 60K30 ; 60K25 ; 90B18

 \tableofcontents
}

\section{Introduction}
In this work, we study the evolution of a Markov process indexed by a tree in continuous time. The tree can represent a population of cells, polymers or particles. On this population, we consider the evolution of an individual characteristic. This characteristic can represent the size, the age or the rate of a nutriment. During the life of an individual,  its characteristic evolves according to an underlying Markov process. At non-homogeneous time, the individuals die and divide. When one divides, the characteristics of the offspring depend on those and their number. This model was studied in \cite{ABBZ11,GW,Feller,Bertoin,EHK10,Harris}. Here, we study the asymptotic behaviour of the empirical measure which describes the population. Following \cite{GW}, we begin to prove a many-to-one formula (or spinal decomposition, tagged fragment ...) and then deduce its long time behaviour. This formula looks like the Wald formula and reduces the problem to the study of a "typical" individual. Closely related, we can find a limit theorem in discrete time in \cite{DM10}, in continuous time with a continuous population in \cite{EW06} and for a space-structured population model in \cite{EHK10}. Our approach is closer to \cite{GW} and extends their law of large number to a variable rate of division. This extension is essential in application \cite{Feller}. In our model, the population is discrete. It is the microscopic version of some deterministic equations studied in \cite{LP,Pert1,Pert2}. Following \cite{FM,TheseT}, we scale our empirical measure and prove that these P.D.E. are macroscopic versions of our model. Before expressing our main results, we begin by giving some notations. If we start with one individual then we will use the Ulam-Harris-Neveu notation \cite{GW}:
\begin{itemize}
\item the first individual is labelled by $\emptyset$;
\item when the individual $u$ divides, then his $K$ descendants are labelled by $u1,...,uK$;
\item we denote by $\mathcal{T}$ the random set of individuals which are dead, alive or will be alive;
\item it is a subset of $\mathcal{U}= \cup_{m\geq0} \ (\N^*)^m$, where $\N=\{0,1,...\}$ and $(\N^*)^0=\{\emptyset\}$;
\item we denote by $\mathcal{V}_t$ the set of individuals which are alive at time $t$;
\item for each $u\in\mathcal{T}$, $\alpha(u)$ and $\beta(u)$ denote respectively the birth and the death date of the individual $u$;
\item we denote by $N_t$ the number of individuals alive at time $t$;
\item for each $u\in \mathcal{T}$ and $t \in [\alpha(u), \beta(u))$, the characteristic of the individual $u$ is denoted by $X^{u}_t$.
\end{itemize}
The dynamics of our model is then as follows.
\begin{itemize}
\item The characteristic of the first individual, $(X^{\emptyset}_t)_{t \in [0, \beta(\emptyset))}$ is distributed according to an underlying c\`adl\`ag strong Markov process $(X_t)_{t\geq 0}$. For sake of simplicity, we will assume that $X=(X_t)_{t\geq 0}$ takes values in a closed subset $E$ of $\R^d$ and is generated by
\begin{equation}
\label{eq:Gen-under}
G f(x) = b(x) \cdot \nabla f(x) + \sigma \Delta f(x),
\end{equation}
where $d\in \N^{*}$, $b:\R^d \rightarrow \R^d$ is a smooth function and $\sigma \in \R_+$. Here, $f$ belongs to the domain $\mathcal{D}(G)$ of $G$. This domain is described in Section \ref{sect:prelim}. Note that our approach is available for another underlying Markov process.

\item The death time $\beta(\emptyset)$ of the first individual verifies
$$
\p\left( \beta(\emptyset) > t \ | \ X^{\emptyset}_s, s \leq t \right)= \int_0^t r(X^{\emptyset}_s) ds,
$$
where $r$ is a non negative, measurable and locally bounded function. Notice that $\alpha(\emptyset)= 0$. 
\item At time $\beta(\emptyset)$, the first individual splits into a random number of children given by an independent random variable $K$ of law
$(p_k)_{k\in \N^*}$. We have $\alpha(0)=...= \alpha(K-1)=\beta(\emptyset)$.
\item We assume that the mean offspring number, which is defined by $m:x\mapsto \sum_{k\geq0} k p_k(x)$, is locally bounded on $E$. 
\item The characteristics of the new individuals are given by $(F^{(K)}_j (X^\emptyset_{\beta(\emptyset)-}, \Theta))_{1\leq j \leq K}$, where $\Theta$ is a uniform variable on $[0,1]$. The sequence $(F^{(k)}_j)_{j\leq k , k \in \N^*}$ is supposed to be a family of measurable functions. 
\item Finally, the children evolve independently from each other like the first individual.
\end{itemize}
The last point is the branching property. To obtain a limit theorem, we follow the approach of \cite{GW}. In this paper, the cell's death rate $r$ and the law of the number of descendants $(p_k)_{k \geq 1}$ are constant. A many-to-one formula is proved:
\begin{equation}
\label{MTOintro}
\frac{1}{\E[N_t]} \E\left[\sum_{u \in \mathcal{V}_t} f(X^u_t)\right] = \E[f(Y_t)],
\end{equation}
where $Y$ is generated, for any smooth function $f$ and $x\in E$, by
\begin{equation}
\label{Gcst}
 A_0 f(x) = G f(x) + r m \sum_{k \geq 1} \frac{k p_k}{m} \int_0^1 \left( \frac{1}{k} \sum_{j=1}^k  f(F_j^{(k)}(x,\theta)) - f(x)  \right) d\theta.
\end{equation}
This process evolves as $X$, until it jumps, at an exponential time with mean $1/rm$. We observe that $r$ is not the jump rate of the auxiliary process. There is a biased phenomenon. It is described in \cite{GW,Harris} and their references. We can interpret it by the fact that the faster the cells divide, the more descendants they have. That is why a uniformly chosen individual has an accelerated
rate of division. A possible generalisation of (\ref{MTOintro}) is a Feynman-Kac formula as in \cite{Harris}:
\begin{equation}
%\label{MTO}
\E\left[\sum_{u\in \mathcal{V}_t} f(X^u_t)\right]=\E\left[f(Y_t) e^{\int_0^t r (Y_s) (m(Y_s)-1) ds}\right] \nonumber,
\end{equation} 
where $Y$ is an auxiliary process starting from $x_0$ and generated by \eqref{Gcst}. Using a Poisson point process, in \cite{Feller}, we get also another representation of the empirical measure to prove the extinction of a parasite population. However, it is difficult to exploit these formulas. Inspired by \cite{EW06,LP,Pert1,Pert2}, we follow an alternative approach. In (\ref{MTOintro}), $Y$ can be understood as a uniformly chosen individual. The problem is: if $r$ is not constant then a uniformly chosen individual does not follow Markovian dynamics. Our solution is to choose this individual with an appropriate weight. This weight is the eigenvector $V$ of the following operator:
$$
 \mathcal{G} f(x)=  G f (x) + r (x)  \left[ \left( \sum_{k \geq 0}  \sum_{j=1}^k \int_0^1 \ f(F^{(k)}_j (x,\theta)) d\theta p_k(x) \right) - f(x) \right].
$$
It is not the generator of a Markov process on $E$. It is described in the next section. Under some assumptions, we are able to prove the following weighted many-to-one formula:
\begin{equation}
\label{MTOO}
\frac{1}{\E[\sum_{u \in \mathcal{V}_t} V(X^u_t)]} \E\left[\sum_{u \in \mathcal{V}_t} f(X^u_t) V(X^u_t)\right] = \E[f (Y_t)],
\end{equation}
where $Y$ is an auxiliary Markov process, starting from $x_0$. It is generated by $A = M + J $, where $M$ describes the motion between the jumps and is defined by
$$
Mf(x)= \frac{G(f \times V)(x) - f(x)GV(x)}{V(x)} =  Gf(x) + \sigma\frac{\nabla V(x). \nabla f(x)}{V(x)},
$$
and $J$ describes the jump dynamics and is given by
\begin{equation*}
J f(x) = \Lambda (x) \left[\frac{\sum_{k \in \N} \sum_{j=1}^k \int_0^1 V\left(F_j^{(k)}(x, \theta)\right)  f\left(F_j^{(k)}(x, \theta)\right) d\theta p_k(x)}{\sum_{k \in \N} \sum_{j=1}^k \int_0^1 V\left(F_j^{(k)}(x, \theta) \right) d\theta p_k(x)} - f(x) \right],
\end{equation*}
where
$$
\Lambda (x)= \left[\sum_{k \in \N} \sum_{j=1}^k \int_0^1 V\left(F_j^{(k)}(x, \theta) \right) \ d\theta \ p_k(x)\right] \times \frac{r (x)}{V(x)}.
$$  
These formulas seem to be complicated but they are very simple when applied. We also observe a biased phenomenon. But contrary to the previous formulas, the bias is present in the motion and the branching mechanism. This bias has been already observed in another context \cite{EW06}. Also note that we do not assume that $\lambda_0$ is the first eigenvalue. So, it is possible to have different many-to-one formulas as can be seen in Remark \ref{rq:plusieurvp}. We can find some criteria for existence of eigenelements in \cite{BCG12,valprop,vpmito,Pinsky} and theirs references. If $Y$ is ergodic with invariant measure $\pi$ then Formula \eqref{MTOO} gives
$$
\lim_{t \rightarrow + \infty} \frac{1}{\E\left[ \sum_{u \in \mathcal{V}_t} V(X^u_t)\right]} \E\left[ \sum_{u \in \mathcal{V}_t} f(X^u_t) V(X^u_t) \right] = \int f \ d\pi,
$$
for all continuous and bounded function $f$. We improve this result:
\begin{theo}[Long time behaviour of the empirical measure]
\label{thintro}
If the following assumptions holds,
\begin{itemize}
\item $X^\emptyset_0$ is deterministic;
\item the system is non explosive; namely $N_t<+\infty$ a.s. for all $t\geq0$;
\item there exists $(V,\lambda_0)$ such that $\mathcal{G}V = \lambda_0 V$ and $V>0$;
\item $Y$ is ergodic with invariant measure $\pi$;
\end{itemize}
then for any measurable function $g$ such that:
\begin{itemize}
 \item there exists $C>0$, such that for all $x \in E$, $|g(x)| \leq C V(x)$;
 \item there exists $\alpha < \lambda_0$, such that $\E [ V^2(Y_t)] \leq C e^{\alpha t}$ and 
 $$
 \E \left[ \frac{r(Y_t)}{V(Y_t)} \int_0^1 \sum_{\underset{a \neq b}{\overset{}{a,b \in \N^*}}} \sum_{k \geq \max(a,b)} p_k(Y_s)  V\left( F_a^{(k)} (Y_s, \theta)\right)  V\left( F_b^{(k)} (Y_s, \theta)\right) d\theta \right] \leq C e^{\alpha t};
 $$
\end{itemize}
then we have
$$
\lim_{t \rightarrow + \infty} e^{- \lambda_0 t} \sum_{u \in \mathcal{V}_t} g (X^u_t) = W \int \frac{g}{V} \ d \pi,
$$
where $W= \lim_{t \rightarrow + \infty} e^{- \lambda_0 t} V(x_0)^{-1} \sum_{u\in \mathcal{V}_t} V (X^u_t)$ and the convergences hold in probability. If furthermore $V$ is lower bounded by a positive constant then
$$
 \lim_{t \rightarrow + \infty} \frac{\1_{W \neq 0}}{N_t} \sum_{u \in \mathcal{V}_t} g(X^u_t) = \1_{W \neq 0} \int \frac{g}{V} \ d\pi / \int \frac{1}{V} \ d\pi \text{ in probability }.
$$
\end{theo} 
If $r$ and $(p_k)_{k\in \N}$ are constant then $V\equiv 1$ is an eigenvector, and so this theorem generalises \cite[Theorem 1.1]{GW}. On the other hand, our model is microscopic and is a scaled version of some deterministic models. More precisely, let $(\mathbf{Z}_t)_{t\geq0}$ be the empirical measure. It is defined, for all $t\geq0$, by
 $$
 \mathbf{Z}_t= \sum_{u\in \mathcal{V}_t} \delta_{X^u_t}.
 $$
Now, let $\mathbf{Z}^{(n)}$ be distributed as $\mathbf{Z}$ and let us consider the following scaling $\mathbf{X}^{(n)} = \frac{1}{n} \mathbf{Z}^{(n)}$. We have: 

\begin{theo}[Law of large number for the large population]
\label{th:Grandpopintro}
If the following assumptions hold
\begin{itemize}
\item $T>0$; 
\item $r$ is upper bounded;
\item there exist $\bar{k}\geq0$ such that $p_k\equiv0$ for all $k\geq \bar{k}$;
\item either $E\subset \R$ and $F_j^{(k)}(x,\theta) \leq x$ for all $j\leq k$ and $\theta \in [0,1]$ or $E$ is compact;
\item The starting distribution $\mathbf{X}_0^{(n)}$ converges in distribution to $\mathbf{X}_0 \in \mathcal{M}(E)$, embedded with the weak topology;
\item we have
$$
\sup_{n\geq0} \E\left[ \mathbf{X}^{(n)}(E) \right] <+\infty.
$$
\end{itemize}
then $\mathbf{X}^{(n)}$ converges in distribution in $\mathbb{D}([0,T],\mathcal{M}(E))$ to $\mathbf{X}$ which verifies
\begin{equation}
\label{eq:thintro}
 \int_E f(x) \ \mathbf{X}_t (dx)  = \int_E f(x) \ \mathbf{X}_0 (dx) + \int_0^t \int_E \mathcal{G} f (x) \ \mathbf{X}_s (dx) ds.
\end{equation}
\end{theo}
Here, $\mathbb{D}([0,T],\mathcal{M}(E))$ is the space of c\`ad-l\`ag functions embedded with the skohorod topology \cite{B99,critere}. We observe that if $\mathbf{X}_0$ is deterministic then $\mathbf{X}_t$ is deterministic for any time $t\geq0$. The equation \eqref{eq:thintro} can be written as
$$
\partial_t n(t,x) + \nabla \left( b(x) n(t,x) \right) + r(x) n(t,x) = \sigma \partial_{xx} n(t,x) + \sum_{k \geq 1}  \sum_{j=1}^k K_j^k \left(r \times p_k \times n(t,\cdot) \right).
$$
where $\mathbf{X}_t= n(t,x)dx$ and $K_j^k$ is the adjoint operator of $f \mapsto \int_0^1 f(F_j^{(k)} (x, \theta)d\theta$. This equation was studied in \cite{LP,Pert1,Pert2} and Theorem \ref{thintro} is relatively close to their limit theorems. We will see in the next section that it is also the Kolmogorov equation associated to $\mathbf{Z}$. So, we observe that $\mathbf{X}$ is equal to the mean measure of $\mathbf{Z}$; that is $f \mapsto \E[\int_E f(x) \ \mathbf{Z}_t(dx)]$. This average phenomenon comes from the branching property. After a branching event, each cell evolves independently from each other, there is not interaction or mutation. Another reason is the linearity of the operator $\mathcal{G}$. From the many-to-one formula, we also deduce that, in large population, the empirical measure behaves as the auxiliary process. The proof is based on the Aldous-Rebolledo criterion \cite{critere,Roel} and is inspired by \cite{FM,dernlemm,TheseT}.\\

In the end of the paper, these two theorems are applied to some structured population models. Our main example is a size-structured population. In this example, the size of cells grows linearly and if a cell dies then it divides to two descendants. Thus, there is motion between the branching events and discontinuity during division. This model is a branching version of the well known TCP windows size process \cite{TCP,GRZ,lcst,Ott2}. For this example, we are able to give some explicit formulas of the invariant distribution, the moments or the rate of convergence. We also prove that, in large population, the empirical measure behaves according to the deterministic equation \eqref{eq:thintro} plus a Gaussian noise.\\

\textbf{Outline.} %The remainder of the paper is organized as follows. 
In the next section, we introduce some  properties of the empirical measure. In Section \ref{secttpslg}, we focus our interest on the long time behaviour. We prove some many-to-one formulas and deduce a general limit theorem which implies Theorem \ref{thintro}. Section \ref{sectlargepop} is devoted to the study of large populations. In this one, we prove Theorem \ref{th:Grandpopintro}. Note that Section \ref{secttpslg} and Section \ref{sectlargepop} are independent. In Section \ref{exemple}, we give our main example, which describes the cell mitosis. Moreover, we give two theorems for the long time behaviour of our empirical measure in addition to some explicit formulas. We also give a central limit theorem for asymmetric cell division for the macroscopic limit. In section \ref{sect:autres-exemples}, we finish by two classical examples which are branching diffusions and self-similar fragmentation.

\section{Preliminaries}
\label{sect:prelim}

In this section, we describe a little more the empirical measure $(\mathbf{Z}_t)_{t\geq0}$. We recall that
$$ \forall t\geq0, \ \mathbf{Z}_t= \sum_{u \in \mathcal{V}_t} \delta_{X^{u}_t}.$$
It belongs to in the space $\mathbb{D}(\R_+,\mathcal{M}(E))$ of c\`ad-l\`ag functions with values in $\mathcal{M}(E)$, which is the set of finite measures on $E$. Let us add the following notations:
$$
\mathbf{Z}_t(f) = \int_E f(x) \mathbf{Z}_t(dx)= \sum_{u\in \mathcal{V}_t} f(X^u_t),
$$
and
$$
\mathbf{Z}_t(1+x^p)= \int_E 1 + x^p \ \mathbf{Z}_t(dx)= \sum_{u\in \mathcal{V}_t} 1 + (X^u_t)^p.
$$
We can describe the dynamics of the population with a stochastic differential equation. That is given, for any smooth function $f:(x,t) \mapsto f(x,t) =f_t(x)$ on $E\times \R_+$, by
\begin{align*}
\mathbf{Z}_t (f_t) 
=& \mathbf{Z}_0 (f_0) + \int_0^t \int_E  \mathcal{G} f_s(x) + \partial_t f_s(x) \mathbf{Z}_s (dx) ds \\
+& \int_0^t \sum_{u \in V_s} \sqrt{2 \sigma} \partial_x f_s(X^u_s) d B^u_s \\
+& \int_0^t \int_{\mathcal{U} \times \R_{+} \times \N^* \times [0,1]} [ \1_{ \{ u \in V_{s-}, l \leq r(X^u_{s-})\}} \\
&  ( \sum_{j=1}^k f_s( F^{(k)}_j (X^u_{s-},\theta))  - f_s(X^u_{s-}) ) ] \  \rho (d s, d u, d l, d k , d \theta ), 
\end{align*}
where $(B^u)_{u\in \mathcal{U}}$ is a family of independent standard Brownian motions and $\rho (d s, d u, d l, d k , d \theta )$ is Poisson point process on $\R_{+} \times \mathcal{U} \times \R_{+} \times \N^* \times [0,1]$ of intensity 
$$
\bar{\rho} (d s, d u, d l, d k , d \theta ) = ds \ n(du) \ dl \ dp_k \ d\theta.
$$
It is also independent from the Brownian motions. We have denoted by $n(du)$ the counting measure on $\mathcal{U}$ and $ds, \ dl, \ d\theta$ are Lebesgue measures. A necessary and sufficient condition for the existence of our process is the non-explosion of $\mathbf{Z}$:
\begin{hypo}[Non explosion]
\label{hyp:nonexplo}
For all $t\geq0$, $N_t < + \infty$ a.s.. 
\end{hypo}
 For instance, we have
\begin{lem}[Sufficient condition to non explosion]
\label{lem:rkborne}
If $r \leq \bar{r}$ and $p_k = 0$ for all $k\geq \bar{k}$, where $\bar{r},\bar{k}>0$, then Assumption \ref{hyp:nonexplo} holds. Moreover, for any $T>0$, we have
$$
\forall t\leq T, \ \E[N_t] \leq \E[N_0] \ \mbox{e}^{(\bar{k}-1) \bar{r} T}.
$$
\end{lem}
\begin{proof}
In this case, we can bound $N_t$ by a branching process independent of the underlying dynamics.
\end{proof}

If there is no explosion then we have a semi-martingale decomposition. Before giving it, we give some precisions and notations about the domain $\mathcal{D}(G)$ of $G$. If $G$ is given by \eqref{eq:Gen-under} then we have $\mathcal{D}(G) = \mathcal{C}^2_b(E,\R)$, which is the set of bounded and $\mathcal{C}^2$ functions with bounded derivatives. We denote by $\mathcal{D}_t(G)$ the domain of the Markov process $(X_t,t)_{t\geq0}$ . Also, we have $\mathcal{D}_t(G) = \mathcal{C}^{1,2}_b(\R_+ \times E,\R)$. It is the set of bounded functions, which are $\mathcal{C}^1$ in their first coordinate, $\mathcal{C}^2$ in their second, and which have bounded derivatives. Finally we denote by $\mathcal{D}(G^2)$ (resp. $\mathcal{D}_t(G^2)$ ) the set of function $f\in \mathcal{D}(G)$ (resp. $f\in\mathcal{D}_t(G)$ ) such that $f^2 \in \mathcal{D}(G)$(resp. $f\in\mathcal{D}_t(G^2)$ ) .

\begin{lem}[Semi-martingale Decomposition]
\label{semimart}
If Assumption \ref{hyp:nonexplo} holds, then for all bounded $f=(f_t)_{t\geq0} \in \mathcal{D}_t(G)$ and $t \geq 0$, we have
$$
 \mathbf{Z}_t(f_t) = \mathbf{Z}_0(f_0) + \mathbf{M}_t(f) + \mathbf{V}_t(f) 
$$
where
\begin{align*}
\mathbf{V}_t(f) &= \int_0^t \mathbf{Z}_s ( \mathcal{G} f_s + \partial_s f_s)  ds,
\end{align*}
and if $f\in\mathcal{D}(G^2)$ then the bracket of $ \mathbf{M}_t(f) $ is given by
\begin{align*}
\langle \mathbf{M}(f) \rangle_t 
&= \int_0^t G(f_s^2)(x)- 2f_s(x) Gf_s(x) \mathbf{Z}_s (dx)\\
&+ \int_E r (x) \int_0^1 \sum_{k \in \N^*} \left(\sum_{j=1}^k f_s(F^{(k)}_j \left(x,\theta)\right)-f_s(x)\right)^2  p_k(x)  d\theta  \mathbf{Z}_s (dx) ds
\end{align*}
\end{lem}
\begin{proof}
It is an application of Dynkin and It\^o formulas, see for instance \cite[Lemma 3.68 p .487]{Jac} and \cite[Theorem 5.1, p.67]{IW}. 
\end{proof}

We define the mean measure $(\mathbf{z}_t)_{t\geq0}$, for any continuous and bounded function $f$ on $E$, by 
$$
\forall t\geq 0, \ \mathbf{z}_t (f)= \E(\mathbf{Z}_t(f))= \E\left[\sum_{u\in \mathcal{V}_t} f(X^u_t)\right].
$$

\begin{coro}[Evolution equation for the mean measure]
\label{EDP1}
Under Assumption \ref{hyp:nonexplo}, if $f$ is a continuous and bounded function on $E$ then for any $t\geq 0$, we have:
$$
\mathbf{z}_t (f) = \mathbf{z}_0 (f) + \int_0^t \mathbf{z}_s ( G f) + \int_E r (x) \ \sum_{k \geq 1}  \sum_{j=1}^k \int_0^1 \ f\left(F^{(k)}_j (x,\theta)\right) \ d\theta \ p_k(x) - f(x) \ \mathbf{z}_s (dx)  ds.
$$
Furthermore, $(\mathbf{z}_t)_{t\geq0}$ is the unique solution of this integro-differential equation. 
\end{coro}
The previous equation can be written as
$$
\partial_t n(t,x) + \nabla \left( b(x) n(t,x) \right) + r(x) n(t,x) = \sigma \partial_{xx} n(t,x) + \sum_{k \geq 1}  \sum_{j=1}^k K_j^k \left(r \times p_k \times n(t,\cdot) \right).
$$
where $\mathbf{z}_t= n(t,x)dx$ and $K_j^k$ is the adjoint of $f \mapsto \int_0^1 f(F_j^{(k)} (x, \theta)d\theta$.
\begin{proof}
We just have to prove the uniqueness. Consider two probability measure-valued processes $(\mu_t)_{t\geq 0}$ and $(\nu_t)_{t\geq0}$ solution of this P.D.E. with same starting distribution $\mu_0 = \nu_0$. Let $K$ be a compact subset of $E$ and let us consider the following norm
$$
\Vert m_1 -m_2 \Vert = \sup_{f \in \mathcal{F}_K} |m_1(f)-m_2(f)| ,
$$
where $f \in \mathcal{C}^2_b(E,\R)$ belongs to $\mathcal{F}_K$ if and only if $\Vert G f \Vert_\infty \leq 1$ and $f(x)=0$ when $x\notin K$. Now, let $f\in  \mathcal{F}_K$, if $x\notin K$ then $\mathcal{G} f(x)=0$. The functions $r$ and $m$ are bounded on $K$ and thus there exists $C_K>0$ such that $\Vert \mathcal{G}f \Vert_\infty \leq C_K$ and then
 \begin{align*}
 |\mu_t(f)-\nu_t(f)| 
 &\leq  C_K \int_0^t  \Vert \mu_s -\nu_s \Vert ds.
 \end{align*}
 Taking the supremum and using Gronwall Lemma we deduce that $\mu_t(f) =\nu_t (f)$ for any $f \in \mathcal{F}_K$. As there is no explosion, we can approach any continuous and bounded function by dominated convergence, and then the equality holds for any continuous and bounded function.
\end{proof}

\section{Long time's behaviour}

\label{secttpslg}
\label{longtime}
Let us recall that
$$
\mathcal{G} f(x)=  G f (x) + r (x) \left[ \sum_{k \geq 0}  \sum_{j=1}^k \int_0^1 \ f(F^{(k)}_j (x,\theta)) \ d\theta \ p_k(x) - f(x) \right].
$$
In the following, we will prove some formulas which characterise the mean behaviour of our model. Then we will use them to prove our limit theorems. 

\subsection{Eigenelements and auxiliary process}

As said in introduction, the existence of eigenelements is fundamental in our approach. Henceforth, we assume the following.

\begin{hypo}[Existence of eigenelements]
\label{hyp:vp}
Assumption \ref{hyp:nonexplo} holds, and there exist $\lambda_0>0$ and $V \in \mathcal{D}(G)$ such that $V>0$ and $\mathcal{G} V = \lambda_0 V$.
\end{hypo}

Under this assumption, we introduce the martingale $(\mathbf{Z}_t(V) e^{-\lambda_0 t})_{t\geq0}$ which plays an important role in the proof of theorem \ref{thintro}.

\begin{lem}[Martingale properties]
\label{lem:Vmart}
If Assumption \ref{hyp:vp} holds  and 
$$
\mathbf{z}_0(V) <+ \infty,
$$
then the process $(\mathbf{Z}_t(V)e^{-\lambda_0 t})_{t \geq 0}$ is a martingale. Moreover, it converges  almost surely to a random variable $W$. 
\end{lem}
\begin{proof}
First, by corollary \ref{EDP1}, we have
\begin{align*}
\mathbf{z}_t(V)
&=\mathbf{z}_0(V) + \int_0^t \mathbf{z}_s(\mathcal{G}V) ds \\
&=\mathbf{z}_0(V) + \lambda_0 \int_0^t \mathbf{z}_s(V) ds.
\end{align*}
Hence for all $t\geq0$, we have $\mathbf{z}_t(V)= \mathbf{z}_0 e^{\lambda_0 t}$. Then if $\mathcal{F}_t= \sigma\{ Z_s \ | \ s \leq t  \}$ then  the Markov properties, applied on $\mathbf{Z}$, gives
$$
\E[\mathbf{Z}_{t+s}(V) | \mathcal{F}_s] = \E[\widetilde{\mathbf{Z}}_t (V) | \widetilde{\mathbf{Z}}_0 = \mathbf{Z}_s ],
$$
where $\widetilde{\mathbf{Z}}$ is distributed as $\mathbf{Z}$. Then $\E[\mathbf{Z}_{t+s}(V) | \mathcal{F}_s] = \mathbf{Z}_s(V) e^{\lambda_0 t}$ and thus
$$
\E[\mathbf{Z}_{t+s}(V) e^{-\lambda_0 (t+s)} \ | \ \mathcal{F}_s] = \mathbf{Z}_s(V) e^{\lambda_0 s}.
$$
Since $(\mathbf{Z}_t(V)e^{-\lambda_0 t})_{t \geq 0}$ is a positive martingale, it converges almost surely.
\end{proof}

\begin{lem}[Weighted many-to-one formula]
\label{lem:WMTO}
Under Assumption \ref{hyp:vp}, if $\mathbf{Z}_0=\delta_{x_0}$, where $x_0 \in E$, then we have
\begin{equation}
\label{eq:WMTO}
 \frac{1}{\E\left[\sum_{u \in \mathcal{V}_t} V(X^u_t) \right]} \E\left[\sum_{u \in \mathcal{V}_t} V(X^u_t) f(X^u_t,t) \right] = \E[f(Y_t,t) \ | \ Y_0 =x_0],
\end{equation}
for any non negative function $f$ on $E\times \R_+$ and $t\geq0$, where $Y$ is a Markov process generated by $A$.
\end{lem}

\begin{proof}
If $\gamma_t : f \mapsto \mathbf{z}_t(f \times V) e^{- \lambda_0 t} V(x_0)^{-1}$ then for all $t\geq 0$ and smooth $f$ we have 
$$
\partial_t \gamma_t (f) = \mathbf{z}_t ( \mathcal{G}(V f) + V \partial_t f - f \mathcal{G} V  )e^{- \lambda_0 t} V(x_0)^{-1} =\gamma_t ( A f + \partial_t f).
$$
Now, by Dynkin formula, the right hand side of \eqref{eq:WMTO} verifies the same equation. The uniqueness can be proved as in the corollary \ref{EDP1}.
\end{proof}

\begin{Rq}[The first characteristic can be random]
If $\mathbf{Z}_0=\delta_{X_0^\emptyset}$, where $X_0^\emptyset$ is random and distributed according to a probability measure $\mu$, then \eqref{eq:WMTO} holds, where $Y$ starts form $Y_0$ which is distributed according to $\mu$.
\end{Rq}

\begin{Rq}[Schr\"{o}dinger operator and $h-$transform]
$\mathcal{G}$ is not a Markov generator. As we have, for all $f$ smooth enough,
$$
\mathcal{G} f = B f - r(m-1) f,
$$
where $B$ is a Markov generator and $\mathcal{G}$ a so-called a Schr\"{o}dinger operator. Its study is connected to the Feynman-Kac formula. The key point of our weighted many-to-one formula is an $h-$transform (Girsanov type transformation) of the Feynman-Kac semigroup as in \cite{P95}. This transformation is usual in the superprocesses study \cite{EW06}.
\end{Rq}

\begin{Rq}[Galton-Watson tree and Malthus parameter]
\label{Malthus}
If $r$ and $p$ are constant, then $V \equiv 1$ is an eigenvector with respect to the eigenvalue $\lambda_0=r (m-1)$, where $m= \sum_{k\geq0} k p_k$ denotes the mean offspring number. So, $\mathbf{Z}_t (V) = N_t$ and the population grows exponentially. This result is already know for $N_t$. It is a continuous branching process \cite{branch,GW}. Furthermore, since Thomas Malthus (1766-1834) has introduced the following simple model to describe the population evolution:
$$
\partial_t N_t = \text{ birth} - \text{death} = b N_t - d N_t = \lambda_0 N_t \implies N_t = e^{\lambda_0 t},
$$ 
in biology and genetic population study, $\lambda_0$ is sometimes called the Malthus parameter.
\end{Rq}

\begin{Rq}[Many eigenelements are possible!]
\label{rq:plusieurvp}
In the previous lemmas, $\lambda_0$ was not required to be the first eigenvalue. So, it is possible to have different eigenelements and auxiliary processes. Consider the example of \cite{Feller}, where some eigenelements are explicit; that is:
$$
\forall x > 0,  \ Gf(x) =  a x f'(x) +  b x f''(x),
$$
where $a,b$ are two non-negative numbers. We also consider that $p_2=2$ and for all $j\in \{1,2\}$,
$$
 \E[f(F^{(2)}_j (x,\Theta))] = \E\left[f(H x)\right],
$$
where $H$ is a symmetric random variable on $[0,1]$ i.e. $H\egalloi 1-H$. This example models cell division with parasite infection. In this case,
$$
\mathcal{G} f(x) = a x f'(x) +  b x f''(x) + r(x) \left( 2\E[f(Hx)] - f(x) \right),
$$
where $a$ is an eigenvalue of $\mathcal{G}$ and $V(x)=x$ is its eigenvector. So, we should have
$$
\E\left[\sum_{u \in \mathcal{V}_t} X^u_t f(X^u_t)\right] = \E[f(Y_t)]e^{at}x_0,
$$
where $Y$ is a Markov process, generated by
$$
G_Y f(x) = \left( a x + 2 b\right) f' (x) + b x f''(x) + r(x) \left( 2 \E[H f(H x)] - f(x)\right).
$$
We can see a bias in the drift terms and jumps mechanism which is not observed in \cite{GW,Feller}. When $r$ is affine, we obtain a second formula. Indeed, if $r(x)=c x + d$, with $c \geq 0$ and $d > a$ (or $d > 0$ and $c=0$) then $V_1(x) = x (c/(d-a)) + 1$ is an eigenvector with respect to the eigenvalue $\lambda_1= d $ ($\Rightarrow \lambda_1 > \lambda_0=a$). Thus, we should also write
$$
\E\left[\sum_{u \in \mathcal{V}_t} f(X^u_t)\right] e^{- d t} =  \ \E\left[\frac{f(U_t)}{ \tau U_t + 1 }\right](\tau x_0 + 1),
$$
where $\tau = \frac{c}{d -a}$ and $U$ is generated by 
$$
G_U f(x)= \left(ax + \frac{2 b x \tau}{\tau x +1 } \right)  f'(x) + b x f''(x) + \frac{r(x)(\tau x + 2)}{ \tau x + 1} \left( \frac{2\E[ ( \tau H x + 1) f( H x )]}{ \tau x + 2} - f(x) \right).
$$
\end{Rq}

\subsection{Many-to-one formulas}
In order to compute our limit theorem, we need to control the second moment. As in \cite{GW}, we begin by describing the population over the whole tree. Then we give a many-to-one formula for forks. Let $\mathcal{T} $ be the random set representing cells that have lived at a certain moment. It is defined by
$$
\mathcal{T} = \lbrace u \in \mathcal{U} \ | \ \exists t>0 , u \in \mathcal{V}_t \rbrace.
$$ 
Lemmas \ref{3.5} and \ref{3.9}, that follow, are respectively the generalisation of \cite[proposition 3.5]{GW} and \cite[proposition 3.9]{GW}.
\begin{lem}[Many-to-one formula over the whole tree]
\label{3.5}
Under Assumption \ref{hyp:vp}, if $\mathbf{Z}_0=\delta_{x_0}$, where $x_0 \in E$, then for any non-negative measurable function $f : E \times \R_+ \rightarrow \R$, we have
$$
\E\left[\sum_{u \in \mathcal{T}} f\left(X^u_{\beta(u)-},\beta(u)\right)\right]= \int_0^{+ \infty} \E\left[ f(Y_s,s) \ \frac{r(Y_s)}{V(Y_s)}  \right] V(x_0)e^{\lambda_0 s} ds
$$
\end{lem}
\begin{proof}
First we have, for all $u\in \mathcal{U}$,
$$
\E\left[ \1_{ \{ u \in \mathcal{T}\}} f\left(X^u_{\beta(u)-},\beta(u)\right)\right] = \E\left[ \1_{ \{ u \in \mathcal{T} \}} \int_{\alpha(u)}^{\beta(u)} f(X^u_s,s) r(X^u_s) ds \right]
$$
because
\begin{align*}
&\E\left[\1_{\{ u \in \mathcal{T} \} } \int_{\alpha(u)}^{\beta(u)} f(X^u_s,s) r(X^u_s) ds\right]\\
=&\E\left[\1_{\{u \in \mathcal{T} \} } \int_0^{+\infty} \int_{\alpha(u)}^\tau  f(X^u_s,s) r(X^u_s) ds \ r (X^u_\tau) e^{- \int_{\alpha(u)}^\tau r(X^u_t) dt} \ d\tau\right]\\
=&\E\left[\1_{\{u \in \mathcal{T} \} } \int_{\alpha(u)}^{+\infty} \int_s^{+\infty} r (X^u_\tau) e^{- \int_{\alpha(u)}^\tau r(X^u_t) dt} \ d\tau \  f(X^u_s,s)r(X^u_s) ds\right]\\
=&\E\left[\1_{\{u \in \mathcal{T}\}} \int_{\alpha(u)}^{+\infty} e^{- \int_{\alpha(u)}^s r(X^u_t) dt}   f(X^u_s,s) r(X^u_s) ds\right]\\
=&\E\left[ \1_{\{u \in \mathcal{T}\}} f\left(X^u_{\beta(u)-},\beta(u)\right)\right].
\end{align*}
Thus,
\begin{align*}
\E\left[\1_{\{u \in \mathcal{T} \}} f\left(X^u_{\beta(u)-},\beta(u)\right)\right]
&= \E\left[ \int_0^{+ \infty} \1_{\{u \in V_s \}} f(X^u_s,s) r(X^u_s) ds \right], \\ 
\end{align*}
and finally,
\begin{align*}
\E\left[\sum_{u \in \mathcal{T}} f\left(X^u_{\beta(u)-},\beta(u)\right)\right]
&= \int_0^{+ \infty} \E\left[  \sum_{u \in V_s} f(X^u_s,s) r(X^u_s)  \right] ds \\
&= \int_0^{+ \infty}  \ \E\left[ f(Y_s,s) \frac{r(Y_s)}{V(Y_s)}  \right] \ V(x_0)e^{\lambda_0 s} ds.
\end{align*}

\end{proof}

If we set $g(x,s)= f(x,s)/V(x)$ then we have:
$$
\E\left[ \sum_{u \in \mathcal{T}} g \left( X^u_{\beta(u)-}, \beta(u)\right) \ V\left(X^u_{\beta(u)-}\right) \right] = \int_0^{+ \infty} \E \left[ g( Y_s,s ) r(Y_s) \right] \times \E [\mathbf{Z}_s(V)] \ ds.
$$
This equality means that adding the contributions over all the individuals corresponds to integrating the contribution of the auxiliary process over the average number of living individuals at time $s$. Let $(P_t)_{t \geq 0}$ be the semigroup of the auxiliary process; it is defined, for any continuous and bounded $f$, by
$$
P_t f (x)= \E[f(Y_t) \ | \ Y_0=x]
$$

\begin{lem}[Many-to-one formula for forks]
\label{3.9}
Under Assumption \ref{hyp:vp}, if $\mathbf{Z}_0=\delta_{x_0}$, where $x_0 \in E$, then for all non-negative and measurable function $f,g$ on $E$, we have
\begin{align*}
&\E\left[\sum_{u,v \in \mathcal{V}_t, u \neq v} f(X^u_t) V(X^u_t) g(X^v_t) V(X^v_t) \right] \\
= &\E[\mathbf{Z}_t(V)]^2 \int_0^t \frac{1}{\E[\mathbf{Z}_s(V)]} \E\left[ J_2(VP_{t - s}f,VP_{t - s}g)(Y_s) \ \frac{r(Y_s)}{V(Y_s)} \right] ds 
\end{align*}
where $J_2$ is defined by
$$
J_2(f,g)(x)=  \int_0^1 \sum_{a \neq b} \sum_{k  \geq \max(a,b)} p_k(x) f\left( F^{(k)}_a (x, \theta )\right)  g\left( F^{(k)}_b(x, \theta ) \right) \ d\theta.
$$

\end{lem}
The operator $J_2$ describes the starting positions of two siblings picked at random.
\begin{proof}
Let $u,v \in \mathcal{V}_t$ be such that $u \neq v$, then there exists $(w,\tilde{u},\tilde{v}) \in \mathcal{U}^3$ and $a,b \in \N^{*}$, $a\neq b$ such that $u=wa\tilde{u}$ and $v=wb\tilde{v}$. The cell $w$ is sometimes called the most recent common ancestor (MRCA). We have
\begin{align*}
&\E\left[\sum_{u,v \in \mathcal{V}_t, u \neq v} f(X^u_t) V(X^u_t) g(X^v_t) V(X^v_t) \right] \\
= &\sum_{ w \in \mathcal{U}} \sum_{a \neq b} \sum_{\tilde{u},\tilde{v} \in \mathcal{U}}  \E\left[ \1_{\{ u \in \mathcal{V}_t\}} f(X^{u}_t) V(X^{u}_t) \1_{\{v \in \mathcal{V}_t \}} g(X^{v}_t) V(X^{v}_t) \right],
\end{align*}

where $u=wa\tilde{u}$ and $v= wb\tilde{v}$ . Let $ \mathcal{F}_t = \sigma \{\mathbf{Z}_s \ | \ s \leq t \} $. By the conditional independence between descendants, we have

\begin{align*}
&\E\left[\sum_{u,v \in \mathcal{V}_t, u \neq v} f(X^u_t) V(X^u_t) g(X^v_t) V(X^v_t) \right]\\
=&\sum_{ w \in \mathcal{U}} \sum_{a \neq b} \E\left[ \E\left[ \sum_{\tilde{u} \in \mathcal{U}} \1_{\{u \in \mathcal{V}_t\}} f(X^u_t) V(X^u_t) | \mathcal{F}_{\beta(w)}\right] \E\left[ \sum_{\tilde{v} \in \mathcal{U}} \1_{\{ v \in \mathcal{V}_t\}} g(X^v_t) V(X^v_t) | \mathcal{F}_{\beta(w)}\right] \right].
\end{align*}

Therefore, as $\beta (w)$ is a stopping time, then using the strong Markov property and \eqref{eq:WMTO}, we have

\begin{align*}
&\E\left[\sum_{u,v \in \mathcal{V}_t, u \neq v} f(X^u_t) V(X^u_t) g(X^v_t) V(X^v_t) \right]\\
=&\sum_{ w \in \mathcal{U}} \sum_{a \neq b} \E[ \1_{\{wa, wb \in \mathcal{T}, t \geq \beta(w)\}} P_{t - \beta(w)} f( X^{wa}_{\beta(w)}) V(X^{wa}_{\beta(w)})\\
 &P_{t - \beta(w)}g ( X^{wb}_{\beta(w)}) V(X^{wb}_{\beta(w)}) e^{ 2 \lambda_0 (t - \beta(w))}] \\
=&\E\left[ \sum_{w \in \mathcal{T}} \1_{\{t \geq \beta(w)\}} J_2( VP_{t - \beta(w)}f, VP_{t - \beta(w)}g)(X^{w}_{\beta(w)-})  \ e^{ 2 \lambda_0 (t - \beta(w))} \right]\\
=&e^{ 2 \lambda_0 t } V(x_0) \int_0^t \E\left[ J_2(VP_{t - s}f,VP_{t - s}g)(Y_s) \ \frac{r(Y_s)}{V(Y_s)}  \right]e^{- \lambda_0 s} ds.
\end{align*}
\end{proof}

\subsection{Proof of Theorem \ref{thintro}}

In this section, we give the main limit theorem which implies Theorem \ref{thintro}. 

\begin{theo}[General Condition for the convergence of the empirical measure]
\label{Gencv}
Under Assumption \ref{hyp:vp}, if $f$ is a measurable function defined on $E$ and $\mu$ a probability measure such that there exists a probability measure $\pi$, two constants $\alpha < \lambda_0$ and $C>0$, and a measurable function $h$ such that
\begin{enumerate}
\item $\pi( |f|)< + \infty$ and $\forall x \in E, \lim_{ t \rightarrow + \infty} P_t f(x) = \pi (f)$, \label{Pcv} \\
\item $\mu(V) < + \infty$ and $\mu P_t (f^2 \times V) \leq  C e^{\alpha t}$,\label{Pborne} \\
\item $P_t |f| \leq h$ and $\mu P_s \left( J_2(V h,V h) \ \frac{r}{V} \right) \leq C e^{\alpha t}$,\label{eq:Domination}
\end{enumerate}
and $\mathbf{Z}_0 = \delta_{X_0^\emptyset}$, where $X_0^\emptyset \sim \mu$, then we have
$$
\lim_{t \rightarrow +\infty } \frac{1}{\E[\mathbf{Z}_t(V)]}{\sum_{u \in \mathcal{V}_t}} f(X^u_t) V(X^u_t) =  W  \times \pi(f),
$$
where the convergence holds in probability. If furthermore $(\mathbf{Z}_t(V)e^{-\lambda_0 t})_{t\geq0}$ is bounded in $L^2$ then the convergence holds in $L^2$.
\end{theo}

Note that the constants and $\pi$ may depend on $f$ and $\mu$! Also note that $\lambda_0$ is not assumed to be the largest eigenvalue. 

\begin{proof}
As in \cite[Theorem 4.2]{GW}, we first prove the convergence for $f$ such that $\pi (f) =0$. We have $\E[\mathbf{Z}_t(V)] = \mu(V)e^{\lambda_0 t}$ and so
$$
\E \left[\left( \frac{1}{\E[\mathbf{Z}_t(V)]}{\sum_{u \in \mathcal{V}_t}} f(X^u_t) V(X^u_t) \right)^2 \right] = \E\left[ \mathbf{Z}_t (f \times V)^2 e^{- 2 \lambda_0 t} \mu(V)^{-2} \right] = A_t + B_t,
$$
where
$$
A_t =  e^{- 2 \lambda_0 t} \mu(V)^{-2} \E\left[\sum_{u \in \mathcal{V}_t} f^2 (X^u_t) V^2 (X^u_t)\right] = e^{- \lambda_0 t}\mu(V)^{-1} \E\left[ f^2 (Y_t) V(Y_t)\right],
$$
and 
\begin{align*}
B_t &= e^{- 2 \lambda_0 t}\mu(V)^{-2} \E\left[\sum_{u,v \in \mathcal{V}_t , \ u \neq v } f (X^u_t) V(X^u_t) f(X^v_t) V(X^v_t)\right] \\
&= \mu(V)^{-1} \int_0^t \E\left[ J_2(V P_{t - s}f,V P_{t - s}f)(Y_s) \ \frac{r(Y_s)}{V(Y_s)}\right] e^{- \lambda_0 s} ds.
\end{align*}
From \eqref{Pborne}, we get $\lim_{t \rightarrow + \infty} A_t =0$. Since $\pi(f)=0$, from \eqref{Pcv}, we get $\lim_{t\rightarrow} P_t f=0$. Then, by \eqref{eq:Domination} and Lebesgue's theorem, we obtain that, for all $s \geq 0$ and $x \in E$, 
$$
 \lim_{ t \rightarrow + \infty } J_2(V P_{t - s}f,V P_{t - s}f)(x) =0.
$$ 
And again by \eqref{eq:Domination} and Lebesgue's theorem, we obtain that $\lim_{t \rightarrow + \infty} B_t =0$. Now, if $\pi(f) \neq0$ then we have
$$
\mathbf{Z}_t (fV) e^{- \lambda_0 t} - W \pi(f) = \mathbf{Z}_t\left( (f - \pi (f)) V \right) e^{- \lambda_0 t} + \pi(f) \left( \mathbf{Z}_t (V)e^{- \lambda_0} -W \right).
$$
Then, as a consequence of to the first part of the proof, the first term of the sum, in the right hand side, converges to 0 in $L^2$. Moreover, the second term converges to 0 in probability thanks to lemma \ref{lem:Vmart}.
\end{proof}
\begin{proof}[Proof of Theorem \ref{thintro}]
If $f= g/V$ then it is a continuous and bounded function. If $h\equiv1$ then all assumptions of the previous theorem hold and we get the first convergence. Now if $V$ is lower bounded, we can use the same argument with $g=1$ and $f=1/V$ which is also a continuous and bounded function. 
\end{proof}

\section{Macroscopic approximation}
\label{sectlargepop}
To prove Theorem \ref{th:Grandpopintro}, we need to use different topologies on $\mathcal{M}(E)$. Let $(\mathcal{M}(E),d_v)$ (resp.  $(\mathcal{M}(E),d_w)$) be the set of finite measure when it is embedded with the vague (resp. weak) topology. These topologies are defined as follow.
$$
\lim_{n \rightarrow + \infty} d_v (X_n, X_\infty) = 0 \iff \forall f \in C_0 ,\ \lim_{n \rightarrow + \infty} \E[f(X_n)] = \E[f(X_\infty)],
$$ 
$$
\lim_{n \rightarrow + \infty} d_w (X_n, X_\infty) = 0 \iff \forall f \in C_b, \ \lim_{n \rightarrow + \infty} \E[f(X_n)] = \E[f(X_\infty)],
$$ 
where $(X_n)_{\neq0}$ is a sequence of $\mathcal{M}(E)$ and $X_\infty \in \mathcal{M}(E)$. Here, $C_0$ is the set of continuous functions which vanish at infinity, and $C_b$ is the set of continuous and bounded functions. Let $\mathbb{D}([0,T],E)$ and $ C([0,T], E)$ be respectively the set of c\`ad-l\`ag functions embedded with the Skohorod topology and continuous functions embedded with the uniform topology \cite{B99}.

\subsection{Proof of Theorem \ref{th:Grandpopintro}}
Let $(\mathbf{Z}^{(n)})_{n\geq1}$ be a sequence of random measure-valued distributed as $\mathbf{Z}$. In this section, we consider the following scaling: $\mathbf{X}^{(n)} = \frac{1}{n} \mathbf{Z}^{(n)}$, and we describe the behaviour of this scaled process when $n$ goes to infinity. 

To understand the behaviour of our model in a large population, we can consider that it starts from a deterministic probability measure $\mathbf{X}_0$, and approach it by the interesting sequence defined by 
$$
 \mathbf{X}^{(n)}_0 = \frac{1}{n} \sum_{k=0}^n  \delta_{Y_k},
$$
where $(Y_k)_{k \geq 1}$ is a sequence of i.i.d. random variable distributed according to $\mathbf{X}_0$. In other words, we set
$$
 \mathbf{Z}^{(n)}_0 = \sum_{k=0}^n \delta_{Y_k}.
$$
Does $\mathbf{X}^{(n)}$ converge? Yes, it converges. Indeed, by the branching property, we have $\mathbf{Z}^{(n)} \egalloi \sum_{k=0}^n \mathbf{Z}^{Y_k}$, where $\mathbf{Z}^{Y_k}_t$ are i.i.d., distributed as $\mathbf{Z}$ and starting from $ \mathbf{Z}^{Y_k}_0 = \delta_{Y_k}$. Henceforth, if $f$ is a continuous and bounded function then the classical law of large number gives
$$
\forall t \geq 0, \ \lim_{n \rightarrow \infty} \mathbf{X}_t^{(n)} (f) = \E\left[\mathbf{Z}^{Y_1}_t (f)\right] a.s.
$$
So by corollary \ref{EDP1}, it implies that $\mathbf{X}^{(n)}$  (pointwise) converges to the solution $(\mu_t)_{t\geq0}$ of the following integro-differential equation:
\begin{align}
\label{equation}
\mu_t (f) 
&= \mu_0 (f) + \int_0^t \mu_s ( G f) \\
&+ \int_E r(x) \sum_{k \geq 0} p_k(x) \int_0^1 \sum_{j=1}^k f(F^{(k)}_j (x,\theta)) d\theta	-f(x) \mu_s (dx) ds. \nonumber
\end{align}

Theorem \ref{th:Grandpopintro} gives a stronger convergence.

\begin{lem}[Semi-martingale decomposition]
\label{semimart2}
If Assumption \ref{hyp:nonexplo}, then for all $f \in \mathcal{D} (G^2)$ and $t \geq 0$,
$$
 \mathbf{X}^{(n)}_t(f) = \mathbf{X}^{(n)}_0(f) + \mathbf{M}^{(n)}_t(f) + \mathbf{V}^{(n)}_t(f),
$$
where
$$
\mathbf{V}^{(n)}_t(f) = \int_0^t \int_E G f(x)  + r (x) \int_0^1 \sum_{k \in \N} \sum_{j=1}^k f(F^{(k)}_j (x,\theta))-f(x)  p_k d\theta \mathbf{X}^{(n)}_s(dx) ds,
$$
and $ \mathbf{M}^{(n)}_t(f) $ is a square-integrable and c\`adl\`ag martingale. Its bracket is defined by
\begin{align*}
 \langle \mathbf{M}^{(n)}(f) \rangle_t 
 &= \frac{1}{n} \int_0^t 2  \mathbf{X}^{(n)}_s (G f^2) - 2 \mathbf{X}^{(n)}_s ( f \times G f ) \\
 &+ \int_E r (x) \int_0^1 \sum_{k \in \N^*} \left(\sum_{j=1}^k f(F^{(k)}_j (x,\theta))-f(x)\right)^2 p_k(x) d\theta \ \mathbf{X}^{(n)}_s(dx) ds. 
 \end{align*}

\end{lem}
\begin{proof}
It is a direct consequence of Lemma \ref{semimart}. Indeed, if $\mathbb{L}^{(n)}$ is the generator of $\mathbf{X}^{(n)}$ then it verifies
$$
\mathbb{L}^{(n)} F_{f}(\mu) = \restriction{\partial_t \E [ F_{f}( \mathbf{X}^{(n)} ) |\mathbf{X}^{(n)}_0= \mu]}{t=0} = \restriction{\partial_t \E [ F_{f/n}( \mathbf{Z}^{(n)} ) |\mathbf{Z}^{(n)}_0= n \mu]}{t=0} = \mathbb{L} F_{f/n} (n \mu),
$$ 
where $F_f(\mu)=F(\mu(f))$, $F,f$ are two test functions and $\mathbb{L}$ is the generator of $\mathbf{Z}$.
\end{proof}

\begin{Rq}[Non explosion]
Let us recall that, by Lemma \ref{lem:rkborne}, if the assumptions of Theorem \ref{th:Grandpopintro} hold then Assumption \ref{hyp:nonexplo} holds; that is there is no explosion.
\end{Rq}
Let us denote by $\mathcal{L}(U)$ the law of $U$, for any random variable $U$.
\begin{lem}
%, $(\mathbf{X}^{(n)})_{n\geq1}$ is tight for the vague topology.
Under the assumptions of Theorem \ref{th:Grandpopintro} the sequence $(\mathcal{L}(\mathbf{X}^{(n)}))_{n\geq 1}$ is uniformly tight in the space of probability measures on $\mathbb{D}([0,T],(\mathcal{M}(E),d_v))$.
\end{lem}
\begin{proof}
We follow the approach of \cite{FM}. According to \cite{Roel}, it is enough to show that, for any continuous bounded function $f$, the sequence of laws of $\mathbf{X}^{(n)} (f)$ is tight in $\mathbb{D}([0,T],\R)$. To prove it, we will use the Aldous-Rebolledo criterion. Let $C^\infty_c$ be the sef of functions of class $C^\infty$ with finite support, we set $S=C_c^\infty \cup \{\ \1 \}$, where $\1$ is the mapping $x\mapsto 1$. We have to prove that, for any function $f \in S$, we have
\begin{enumerate}
\item $\forall t \geq 0$, $\left(\mathbf{X}^{(n)}_t (f)\right)_{n \geq 0}$ is tight; 
\item for all $n \in \N$, and $\varepsilon, \eta >0$, there exists $\delta>0$ such that for each stopping time $S_n$ bounded by $T$,
$$
\limsup_{n \rightarrow + \infty} \sup_{0 \leq u \leq \delta} \p (| \mathbf{V}^{(n)}_{S_n + u} (f) - \mathbf{V}^{(n)}_{S_n}(f) | \geq \eta) \leq \varepsilon.
$$
$$
\limsup_{n \rightarrow + \infty} \sup_{0 \leq u \leq \delta} \p (| \langle \mathbf{M}^{(n)}(f) \rangle_{S_n + u} - \langle \mathbf{M}^{(n)}(f) \rangle_{S_n} | \geq \eta ) \leq \varepsilon.
$$
\end{enumerate}

The first point is the tightness of the family of time-marginals $(\mathbf{X}^{(n)}_t(f))_{n\geq1}$ and the second point, called the Aldous condition, gives a "stochastic continuity". It looks like the Arzel\`{a}-Ascoli Theorem. Using Lemma \ref{lem:rkborne}, there exists $C>0$ such that
\begin{align*}
\p(|\mathbf{X}^{(n)}_t (f)| > k) &\leq \frac{\Vert f \Vert_{\infty} \ \E [\mathbf{X}^{(n)}_t (\1)] }{k}\\
&\leq \frac{\Vert f \Vert_{\infty} \ C \ \E[\mathbf{X}^{(n)}_0 (\1) ] }{\ k},
\end{align*}
which tends to $0$ as $k$ tends to infinity. This proves the first point.
Let $\delta > 0$, we get for all stopping times $S_n \leq T_n \leq (S_n+\delta) \leq T $, that there exist $C',C_f>0$ such that
\begin{align*}
\E[| \mathbf{V}^{(n)}_{T_n} (f) - \mathbf{V}^{(n)}_{S_n}(f) |]
&= \E\left[\left| \int_{S_n}^{T_n} \mathbf{X}^{(n)}_s ( G f) \right.\right.\\
&+ \left.\left.\int_E r(x) \int_0^1 \sum_{k \in \N} \sum_{j=1}^k f\left(F^{(k)}_j (x,\theta)\right)-f(x)  p_k(x) d\theta \ \mathbf{X}^{(n)}_s(dx) ds \right|\right]\\
&\leq  C' \left[\Vert G f \Vert_{\infty} + \Vert f \Vert_{\infty}\right] \times \E\left[ \left|T_n-S_n\right|\right]  \\
&\leq C_f  \delta.
\end{align*}
In the other hand, there exists $C'_f>0$ such that
\begin{align*}
&\E[| \langle \mathbf{M}^{(n)}(f) \rangle_{T_n} - \langle \mathbf{M}^{(n)}(f) \rangle_{S_n}|] \\
=&\frac{1}{n} \E\left[\left| \int_{S_n}^{T_n} 2  \mathbf{X}^{(n)}_s (G f^2) - 2 \mathbf{X}^{(n)}_s ( f G f ) \right.\right.\\
+&\left.\left. \int_E r(x) \int_0^1 \sum_{k \in \N}  \sum_{j=1}^k \left( f\left(F_j^{(k)}(x,\theta)\right) - f(x) \right)^2 p_k d\theta \ \mathbf{X}^{(n)}_s(dx) ds \right|\right]\\
%\leq &\frac{1}{n}  \times C_{\bar{r}, \bar{k},T ,f } \times (T_n-S_n)\\
\leq & \frac{ C'_f \delta }{n.}
\end{align*}
 Then, for a sufficiently small $\delta$, the second point is verified and we conclude that $(\mathbf{X}^{(n)})_{n\geq 1}$ is uniformly tight in $\mathbb{D}([0,T],(\mathcal{M}(E),d_v))$.
\end{proof} 
\begin{proof}[Proof of Theorem \ref{th:Grandpopintro}] 
Let us denote by $\mathbf{X}$ a limiting process of $(\mathbf{X}^{(n)})_{n\geq 1}$; namely there exists an increasing sequece $(u_n)_{n\geq1}$, on $\N^*$, such that $(\mathbf{X}^{(u_n)})_{n\geq 1}$ converges to $\mathbf{X}$. It is
almost surely continuous in $(\mathcal{M}(E),v)$ since
%We can give another argument, which do not use the M\'el\'eard-Roelly criterion \cite{MR}. As
\begin{equation}
\label{eq:Xcont}
\sup_{t \geq 0} \sup_{ \Vert f \Vert_\infty \leq 1} | \mathbf{X}^{(n)}_{t-}(f) - \mathbf{X}^{(n)}_t (f) | \leq \frac{\bar{k}}{n}.
\end{equation}
%$\mathbf{X}$ is continuous from $[0,T]$ to $(\mathcal{M}(E),d_w)$
In the case where $E$ is compact, the vague and weak topologies coincide. By Doob's inequality, there exists $C>0$ such that
$$
\sup_{f} \E\left[ \sup_{t \leq T} \left| \mathbf{M}_t^{(n)}(f) \right| \right] \leq 2 \sup_{f} \E\left[ \langle \mathbf{M}^{(n)}(f) \rangle_T \right] \leq \frac{ C }{n}
$$
where the supremum is taken over all the function $f \in \mathcal{D}(G^2)$ such that $\Vert f \Vert_\infty \leq 1$. Hence,
\begin{equation}
\label{Mtdzero}
\lim_{n \rightarrow + \infty} \sup_{f} \E\left[ \sup_{t \leq T} \left| \mathbf{M}^{(n)}_t(f) \right| \right] = 0.
\end{equation}
But as
\begin{align*}
\mathbf{M}^{(n)}_t(f) &= \mathbf{X}^{(n)}_t (f) - \mathbf{X}^{(n)}_0 (f)\\
 &- \int_0^t \int_E G f(x)  + r(x) \int_0^1 \sum_{k \in \N} \sum_{j=1}^k f\left(F^{(k)}_j (x,\theta)\right)-f(x) \ p_k(x) d\theta \ \mathbf{X}^{(n)}_s(dx)  ds,
\end{align*}
we have
\begin{align*}
0=&\mathbf{X}_t (f) - \mathbf{X}_0 (f) - \int_0^t \mathbf{X}_s ( G f) \\
+&\int_E r(x)  \left( \sum_{j=1}^k f(F^{(K)}_j (x,\theta)) \ p_k(x) \ d\theta -f(x) \right) \mathbf{X}_s (dx) ds.
\end{align*}
Since this equation has a unique solution, it ends the proof when $E$ is compact. This approach fails in the non-compact case. Nevertheless, we can use the M\'{e}l\'{e}ard-Roelly criterion \cite{MR}. We have to prove that $\mathbf{X}$ is in $C([0,T],(\mathcal{M}(E),w))$ and $\mathbf{X}^{(n)}(\1)$ converges to $\mathbf{X}(\1)$. By \eqref{eq:Xcont}, $\mathbf{X}$ is continuous. To prove that $\mathbf{X}^{(n)}(\1)$ converges to $\mathbf{X}(\1)$, we use the following lemmas.

\begin{lem}[Approximation of indicator functions]

 %$E \subset \R$, \ $| F_j^{(k)} (x, \theta) | \leq | x |$, and for all $k \in \N^*$, 
For each $k\in \N$, there exists $\psi_k\in \mathcal{D}(G)$ such that:
$$
\forall x \in E,\  \1_{[k;+ \infty[} (x)  \leq \psi_k (x) \leq \1_{[k-1;+ \infty[} (x) \ \mbox{and} \ \exists C, \ G \psi_k \leq C \psi_{k-1}.
$$
\end{lem}
\begin{proof}
See \cite[lemma 4.2]{derlemm2} or \cite[lemma 3.3]{dernlemm}.
\end{proof}

\begin{lem}[Commutation of limits]
\label{cvu}
Under the assumptions of Theorem \ref{th:Grandpopintro},
$$
\lim_{k \rightarrow +\infty} \limsup_{n \rightarrow + \infty} \ \E\left[\sup_{t \leq T} \mathbf{X}^{(n)}_t(\psi_k)\right]=0,
$$
where $(\psi_k)_{k \geq 0}$ are defined as in the previous lemma.
\end{lem}
The proof is postponed after. Hence, a same computation to \cite{dernlemm} gives us the convergence in $\mathbb{D}([0,T],(\mathcal{M}(E),w))$. Thus, each subsequence converges to the equation (\ref{equation}). The end of the proof follow with the same argument of the compact case.

We can give another argument, which does not use the M\'el\'eard-Roelly criterion \cite{MR}. As
$$
\sup_{t \geq 0} \sup_{ \Vert f \Vert_\infty \leq 1} | \mathbf{X}^{(n)}_{t-}(f) - \mathbf{X}^{(n)}_t (f) | \leq \frac{\bar{k}}{n},
$$
$\mathbf{X}$ is continuous from $[0,T]$ to $(\mathcal{M}(E),d_w)$. Let $\mathbf{G}$ be a Lipschitz function on $C([0,T],(\mathcal{M}(E),d_w))$, we get, 
\begin{align*}
|\E[\mathbf{G}(\mathbf{X}^{(u_n)})] - \mathbf{G}(\mathbf{X})| %\label{larglim}
& \leq \E\left[ \sup_{t \in [0, T]} d_w \left(\mathbf{X}_t^{(u_n)},\mathbf{X}_t\right) \right] \\
& \ \leq \E\left[ \sup_{t \in [0, T]} d_w\left( \mathbf{X}_t^{(u_n)} , \mathbf{X}_t^{(u_n)}( \cdot \times (1-\psi_k))\right) \right] \\
&  + \E\left[ \sup_{t \in [0, T]} d_w\left( \mathbf{X}_t^{(u_n)}( \cdot \times (1-\psi_k)) , \mathbf{X}_t( \cdot \times (1-\psi_k))\right) \right] \\
&  +  \sup_{t \in [0, T]} d_w \left( \mathbf{X}_t ( \cdot \times (1-\psi_k)) , \mathbf{X}_t \right).
\end{align*}

According to Lemma \ref{cvu}, we obtain that
$$
\lim_{k \rightarrow +\infty} \limsup_{n \rightarrow + \infty} \ \E\left[ \sup_{t \in [0, T]} d_w\left( \mathbf{X}_t^{(u_n)} , \mathbf{X}_t^{(u_n)}( \cdot \times (1-\psi_k))\right) \right]=0
$$ 
and
$$
\lim_{k \rightarrow +\infty} \ \sup_{t \in [0, T]} d_w ( \mathbf{X}_t ( \cdot \times (1-\psi_k)) , \mathbf{X}_t ) =0.
$$
Then, we have 
\begin{align*}
&d_w \left( \mathbf{X}_t^{(u_n)}( \cdot \times (1-\psi_k)) , \mathbf{X}_t( \cdot \times (1-\psi_k))\right) \\
= \ &d_v \left( \mathbf{X}_t^{(u_n)}( \cdot \times (1-\psi_k)) , \mathbf{X}_t( \cdot \times (1-\psi_k))\right).
\end{align*}
Thus,
$$
\lim_{k \rightarrow +\infty} \limsup_{n \rightarrow + \infty} \ \E\left[ \sup_{t \in [0, T]} d_w\left( \mathbf{X}_t^{(u_n)}( \cdot \times (1-\psi_k)) , \mathbf{X}_t( \cdot \times (1-\psi_k))\right) \right] = 0,
$$
by continuity of $\nu \mapsto \nu(1-\psi_k)$ in $\mathbb{D}(\mathcal{M}(E),d_v)$. And finally,
$$
\lim_{n \rightarrow + \infty} \mathbf{G}\left(\mathbf{X}^{(u_n)}\right) = \mathbf{G}(\mathbf{X}),
$$
which completes the proof.
 \end{proof}
\begin{proof}[proof of Lemma \ref{cvu}]
If $\mu_t^{n,k}= \E(\mathbf{X}^{(n)}_t(\psi_k))$ then we have
\begin{align*}
\mu_t^{n,k}
&= \E[\mathbf{X}^{(n)}_0 (\psi_k)] + \int_0^t \E \left[ \int_E G \psi_k(x) \right.\\
&\left.+ r(x) \left( \sum_{k\geq 1} \sum_{j=1}^{k} p_k(x) \int_0^1 \psi_k(F_j^{(k)}(x,\theta)) - \psi_k(x)\right) \mathbf{X}^{(n)}_s(dx) \right] ds\\
&\leq \mu_0^{n,k} + C \int_0^t \mu_s^{n,k-1} + \mu_s^{n,k} ds.\\
\end{align*}
Now, by Gronwall's Lemma, iteration and monotonicity, we deduce that
\begin{align*}
\mu_t^{n,k}
&\leq  C_1 ( \mu_0^{n,k} + \int_0^t \mu_s^{n,k-1} ds)\\
&\leq  C_1 \mu_0^{n,k} + C_1^2 T \mu_0^{n,k-1} + \int_0^t \int_0^s \mu_u^{n,k-2} du ds\\
&\leq   \sum_{l=0}^{k-1}  \mu_0^{n,k-l} C_1 \frac{( C_1 T )^l}{l!} + C_2 \times \frac{( C_1 T )^k}{k!} \\
&\leq    \mu_0^{n,\lfloor k/2 \rfloor} C_1 e^{C_1 T} + C_3 \sum_{l>\lfloor k/2 \rfloor} \frac{( C_1 T )^l}{l!} + C_2 \times \frac{( C_1 T )^k}{k!},
\end{align*}
where $C_1,C_2$ and $C_3$ are three constants. Thus,
$$
\lim_{k \rightarrow + \infty} \limsup_{n \rightarrow + \infty} \mu_t^{n,k}=0.
$$
%Now consider the bracket of the martingale $M_t^{(n)}(\psi_k)$:
%$$
%\E[\langle M_t^{(n)}(\psi_k) \rangle] \leq  \frac{1}{n} ( C_1 \mu_s^{n,k-1} + C_2 \mu_s^{n,k} )
%$$
%because $\psi_k A \psi_k \leq C \psi_{k-1}$. This expression converge to $0$ then by Doob's inequality we get:
%$$
%\lim_{n \rightarrow + \infty} E[\sup_{t \leq T} |M_t^{(n)}(\psi_k)|] =0.
%$$
And finally the following expression completes the proof,
$$
\E\left[\sup_{t \leq T} |\mathbf{X}_t^{n}(\psi_k)|\right] \leq \mu_0^{n,k} + C \int_0^t \mu_s^{n,k-1} + \mu_s^{n,k} ds + \E\left[\sup_{t \leq T} |\mathbf{M}_t^{(n)}(\psi_k)|\right].
$$

\end{proof}
\section{Main example : a size-structured population model}
\label{exemple}

Let us introduce ourf main example. It is a size-structured population model which represents the cell mitosis. It is described as follows: the underlying process $X$ is deterministic and linear and when a cell dies, it divides in two parts. Formally and with our notations, we have
\begin{equation}
\label{notmit1}
E= [0, + \infty), \ Gf= f' \  \text{ and } \ p_2=1,
\end{equation}
\begin{equation}
\label{notmit2}
 \ \forall x \in E, \ \forall \theta \in [0,1], \ F^{(2)}_1 (x,\theta)=F^{-1}(\theta) x\ \text{ and } \ F^{(2)}_2 (x,\theta)= (1-F^{-1}(\theta)) x,
\end{equation} 
where $F$ is the cumulative distribution function of the random variable in $[0, 1]$. It verifies $F(x) = 1-F(1-x)$.
In this case, one cell lineage is generated by:

$$
\forall f \in C^1, \forall x \geq 0, \  L f = f'(x) + r(x) \left[ \E[f(H x)] - f(x) \right],
$$
where $H$ is distributed according to $F$. This process is sometimes called the TCP (Transmission Control Protocol) process in computer science \cite{TCP,GRZ,lcst,Ott2}. Firstly, we prove the non explosion even if $r$ is not bounded.
\begin{lem}[Non explosion]
\label{CS}
Let $p \geq 1$. If for all $x\in \R_+^*$,  $r(x) \leq C_0 (1+x^p)$, and $\mathbf{z}_0(1+x^p)<+\infty$, then our process is no, explosive. Moreover 
$$
\E\left[\sup_{s \in [0, T] } \mathbf{Z}_s (1+x^p)\right] \leq \mathbf{z}_0 (1+x^p) e^{C_p T},
$$
where $C_p$ is constant and $T>0$.
\end{lem}
\begin{proof}
We have
\begin{align*}
\mathbf{Z}_t (f) &= \mathbf{Z}_0 (f) + \int_0^t \int_E f'(x) \ \mathbf{Z}_s (dx) \ ds \\
&+ \int_0^t \int_{\mathcal{U} \times \R_{+} \times [0,1]} \1_{\{u \in V_{s-}, l \leq r(X^u_{s-})\}} f ( \theta X^u_{s-}) + f ( (1-\theta) X^u_{s-})  - f (X^u_{s-})) \  \rho (d s, d u, d l , d \theta )
\end{align*}
Using the same argument to \cite[Theorem 3.1]{FM}, we introduce $\tau_n=\inf \{ \ t\geq 0 \ | \ \mathbf{Z}_t(1+x^p) > n \ \}$; and we have 
\begin{align*}
\sup_{u \in [0, t \wedge \tau_n]} \mathbf{Z}_u (1+x^p) 
\leq &\mathbf{Z}_0 (1+x^p) +  \int_0^{t \wedge \tau_n}  \mathbf{Z}_s (p x^{p-1}) ds \\
+ &\int_0^{t \wedge \tau_n} \int_{\mathcal{U} \times \R_{+} \times [0,1]} \1_{u \in V_{s-}, l \leq r(X^u_{s-})}  \\
 &(1 + ( \theta^p  + (1-\theta)^p-1) (X^u_{s-})^p) \  \rho (d s, d u, d l , d \theta )\\
\leq &\mathbf{Z}_0 (1+x^p) +  \int_0^{t \wedge \tau_n}  p \times \sup_{u \in [0, s \wedge \tau_n] } \mathbf{Z}_u (1+x^p) ds \\
+ &\int_0^t \int_{\mathcal{U} \times \R_{+} \times [0,1]} \1_{\{u \in V_{s-}, l \leq r(X^u_{s-})\}} \  \rho (d s, d u, d l , d \theta ),
\end{align*}
because $( \theta^p  + (1-\theta)^p-1)\leq 0$. Thus there exist $C>0$ such that
\begin{align*}
\E\left[\sup_{u \in [0, t \wedge \tau_n] } \mathbf{Z}_u (1+x^p) \right]
&\leq \mathbf{z}_0 (1+x^p) +  \int_0^{t}  C \ \E \left[ \sup_{u \in [0, s \wedge \tau_n] } \mathbf{Z}_u (1+x^p) \right] ds. \\
\end{align*}
Finally, the Gronwall Lemma implies the existence of $C_p$ such that
$$
\E\left[\sup_{s \in [0, t \wedge \tau_n] } \mathbf{Z}_s (1+x^p)\right] \leq \mathbf{z}_0 (1+x^p) e^{C_p t}.% \leq \mathbf{z}_0 (1+x^p) e^{C_p T}.
$$
We deduce that $\tau_n$ tends almost surely to infinity and that there is non explosion.
\end{proof}

\subsection{Equal mitosis : long time behaviour}
\label{mito}
In this subsection, we establish the long time behaviour of $\mathbf{Z}$. We assume that
$$
\forall x\geq0, \forall\theta\in [0,1], \  F_1^{(2)} (x, \theta) =  F_2^{(2)} (x, \theta) = \frac{x}{2}.
$$
That is, the cells divide in two equal parts. In order to give a many-to-one formula, we recall a theorem of \cite{Pert2}:

\begin{theo}[Sufficient condition for the existence of eigenelement]
\label{vp2}
If there exist $\underline{r}, \bar{r}>0$ such that 
$$
\underline{r} \leq r \leq \bar{r},
$$
$r$ is continuous and $r(x)$ is constant equal to $r_\infty$ for $x$ large enough, then there exists $V\in C^1(\R_+)$ such that
$$
c(1+x^k) \leq V(x) \leq C (1+x^k),
$$
where $C,c$ are two constant and $2^k = \frac{2 r_\infty}{\lambda_0 + r_\infty}$.
\end{theo}
So, we get a many-to-one formula with an auxiliary process generated by

\begin{equation*}
%\label{genb}
Af(x)= f'(x) + r(x) \frac{2 V(x/2)}{V(x)} \left(  f(x/2) - f(x) \right).
\end{equation*}

Our main result gives the two following limit theorems.

\begin{coro}[Convergence of the empirical measure for a mitosis model ]
\label{cvmitointro}
If there exist $\underline{r}, \bar{r}>0$ such that 
$$
 \underline{r} \leq r \leq \bar{r},
$$
$r$ is continuous and $r(x)$ is constant equal to $\bar{r}$ for $x$ large enough, then there exists a probability measure $\pi$ such that, for any continuous and bounded function $g$, we have
$$
\lim_{t \rightarrow + \infty} \frac{1}{N_t} \sum_{u \in \mathcal{V}_t} g(X^u_t) = \int g \ d\pi \text{ in probability}.
$$
In particular for a constant rate $r$, $\pi$ has Lebesgue density:
\begin{equation}
\label{tcpcst}
 x \mapsto \frac{ 2r}{\prod_{n=1}^{ + \infty} (1-2^{-n})} \sum_{n=0}^{+\infty}\left( \prod_{k=1}^{n} \frac{2}{1-2^k}\right) e^{-2^{n+1}  r x}.
\end{equation}
\end{coro}

This explicit formula \eqref{tcpcst} is not new \cite{Pert1,Pert2}, but here, the empirical measure convergences in probability, while in the mentioned papers, the mean measure or the macroscopic process converges (see Theorem \ref{th:Grandpopintro}).

\begin{proof}[Proof of corollary \ref{cvmitointro}]
By Theorem \ref{vp2}, the mapping $x\mapsto V(x/2)/V(x)$ is upper and lower bounded. Thus, the auxiliary process is ergodic and admits a unique invariant law, as can be checked using a suitable Foster-Lyapunov function \cite[Theorem 6.1]{MT} (for instance, we can use $x\mapsto 1 + x$). See also \cite{GK}. Finally, we use Theorem \ref{thintro} to conclude. The explicit formula is an application of \cite{Ott2}.
\end{proof}

We can see that the assumptions of Theorem \ref{vp2} are strong, and not necessary:
\begin{coro}[Convergence of the empirical measure when $r$ is affine]
\label{cvmito}
If 
$$
 \forall x\geq 0, r(x)= a x + b,
$$ 
where $a,b \geq 0$ and $a$ or $b$ is positive then there exists a measure $\pi$ such that
$$
\lim_{t \rightarrow + \infty} \frac{1}{N_t} \sum_{u \in \mathcal{V}_t} g(X^u_t) = \int g \ d\pi.
$$
The convergence holds in probability and for any continuous function $g$ on $E$ such that $\forall x \in E, \ |g(x)| \leq C (1 +x )$.
\end{coro}
\begin{proof}
If $r(x)= a x + b $ then $V(x)= x \frac{\sqrt{b^2 +4 a}- b}{2} +1$ is an eigenvector and $\frac{2 a}{\sqrt{b^2 + 4 a} -b }$ is its corresponding eigenvalue. Henceforth, this result is a direct application of Theorem \ref{thintro}
\end{proof}

\begin{Rq}[Malthus parameter]
We also deduce that
$$
\lim_{t \rightarrow + \infty} N_t \ e^{- \lambda_0 t} = W \int_E \frac{1}{V} \ d\pi,
$$
where $\lambda_0 = \frac{2 a}{\sqrt{b^2 + 4 a} -b }$ is the Malthus parameter (see Remark \ref{Malthus}).
\end{Rq}

\begin{Rq}[Estimation of $r$ for the Escherichia coli cell]
We can find some estimates of the division rate in the literature. An inverse problem was developed and applied with experimental data in \cite{DMZ10}(see also \cite{K69}). More recently, \cite{DHRR11} gives a nonparametric estimation of the division rate.
\end{Rq}

\subsection{Homogeneous case: moment and rate of convergence}
When $r$ is constant, the process is easier to study since the auxiliary process has already been studied \cite{TCP,lcst,Ott2}. Here, we give the moments and a first approach to estimate the rate of convergence.

\begin{lem}[Moments of the empirical measure]
\label{moments}
For all $m \in \N$, and $t \geq 0$, we have
\begin{align*}
\E[ \mathbf{Z}_t (x^m) ] 
&= \E\left[ \sum_{u \in \mathcal{V}_t} (X^u_t)^m \right] \\
&= \int_0^{+ \infty} e^{rt} \left[ \frac{m!}{\prod_{i=1}^m \theta_i} + m! \sum_{i= 1}^m \left( \sum_{k=0}^i \frac{x^k}{k!} \prod_{j=k , j \neq i}^m \frac{1}{\theta_j - \theta_i} \right)  e^{- \theta_i t} \right] \ \mathbf{z}_0 (dx),
\end{align*}
where $\theta_i = 2r \left( 1 - 2^{-i} \right)$. In particular,
\begin{align*}
\E[ \mathbf{Z}_t (x) ] =  \E\left[ \sum_{u \in \mathcal{V}_t^\mu} X^u_t \right] 
&= e^{rt} \int_0^{+ \infty} \frac{1}{r} - \left( \frac{1}{r} - x \right) e^{-r t} \ \mathbf{z}_0 (dx),
%&= \frac{ n}{r} ( e^{rt} - 1 ) + \sum_{i=1}^n x_i
\end{align*}
and
\begin{align*}
\E[ \mathbf{Z}_t (x^2) ] 
&= \E\left[ \sum_{u \in \mathcal{V}_t} (X^u_t)^2 \right] \\
&= e^{rt} \int_0^{+ \infty} \frac{4}{3 r^2} + 2 \left[ e^{-rt} \left( \frac{-2}{r^2} + \frac{2 x}{r} \right) + e^{-3 rt/2} \left( \frac{4}{3 r^2} - \frac{2 x}{3 r} + \frac{x^2}{2} \right) \right] \mathbf{z}_0 (dx).\\
%&=\frac{4 n}{ 3 r^2} \left( e^{rt} - 3 + 2 e^{-rt/2}  \right) + \left( \sum_{i=1}^n x_i \right) \left( \frac{4}{r} - \frac{4}{3r} e^{-rt/2} \right) + e^{-rt/2} \sum_{i=1}^n x_i^2.
\end{align*}
\end{lem}
\begin{proof}
Since $r$ is constant, we have $\mathcal{G}\1= r\1$, where $\1$ is the constant mapping, which is equal to $1$. Lemma \ref{lem:WMTO} gives
$$
\frac{1}{\E[N_t]} \E\left[ \sum_{u\in \mathcal{V}_t} f(X^u_t) \right] = \E[f(Y_t)],
$$
where $Y$ is generated by
$$
A f(x)= f'(x) + 2r \left( f\left(\frac{x}{2} \right) - f(x) \right).
$$
Finally, we complete the proof using \cite[Theorem 4]{lcst}.
\end{proof}
Now, let us talk about the rate of convergence. To estimate the distance between two random measures, we will use the Wasserstein distance\cite{metric,Vil}:
\begin{defi}[Wasserstein distance]
 Let $\mu_1$ and $\mu_2$ two finite measures on a Polish space $(F,d_F)$, the Wasserstein distance between $\mu_1$ and $\mu_2$ is defined by
$$
W_{d_F} (\mu_1, \mu_2) = \inf \int_{F\times F} d_F (x_1,x_2) \Pi(dx_1,dx_2),
$$
where the infimum runs over all the measures $\Pi$ on $F \times F$ with marginals $\mu_1$ and $\mu_2$. In particular, if $\mu_1$ and $\mu_2$ are two probability measures, we have
$$
W_{d_F} (\mu_1, \mu_2) = \inf \E[d_F(X_1,X_2)],
$$
where the infimum runs over all two random variables $X_1,X_2$, which are distributed according to $\mu_1,\mu_2$.
\end{defi}
So, if $M_1, M_2$ are two random measures then
$$
W_d(\mathcal{L}(M_1), \mathcal{L}(M_2)) = \inf \ \E[ d(M_1,M_2) ],
$$
where the infimum is taken over all the couples of random variables $(M_1,M_2)$ such that $M_1 \sim \mathcal{L}(M_1)$ and $M_2 \sim \mathcal{L}(M_2)$, and $d$ is a distance on the measures space. Here, we consider $d=W_{|\cdot|}$. It is the Wasserstein distance on $(E, |\cdot|)$. We have 

\begin{theo}[Quantitative bounds]
\label{vitess}
If $r$ is constant, then we have, for all t $\geq 0$,
\begin{align*}
& W_{W_{|\cdot|}}  \left(\mathcal{L}\left(\frac{\mathbf{Z}^{x}_t}{\E[N_t]}\right) ,\mathcal{L}\left(\frac{\mathbf{Z}^{y}_t}{\E[N_t]}\right)\right) \leq  |x-y| e^{-r t},\\
& W_{W_{|\cdot|}}  \left(\mathcal{L} \left(\frac{\mathbf{Z}^{x}_t}{ N_t}\right) ,\mathcal{L}\left(\frac{\mathbf{Z}^{y}_t}{N_t}\right)\right) \leq |x-y| \frac{rt e^{-rt}}{1- e^{-rt}},\\
\end{align*}
where $\mathbf{Z^x},\mathbf{Z^y}$ are distributed as $\mathbf{Z}$ and start from $\delta_x$ and $\delta_y$. 
\end{theo}
This result does not give a bound for $W_{W_{|\cdot|}} \left(\mathcal{L} \left( \mathbf{Z}_t / \E[N_t] \right), \mathcal{L} \left( W \pi\right)\right)$ or $W_{W_{|\cdot|}} \left(\mathcal{L} \left( \mathbf{Z}_t / N_t \right), \mathcal{L} \left( \pi\right)\right)$, where $\pi$ is the limit measure of Corollary \ref{cvmitointro}.

\begin{proof}[proof of Theorem \ref{vitess}]
By homogeneity, we can see our branching measure $\mathbf{Z}$ as a process indexed by a Galton-Watson tree \cite{GW}. For our coupling, we take two processes indexed by the same tree. More precisely, as the branching time does not depend on the position, we can set the same times to our two processes. Let $\mathcal{T} = \bigcup_{n \in \N} \{1,2\}^n  $ representing cells that have lived at a certain moment. Let $(d_u)_{u \in \mathcal{U}}$ be a family of i.i.d. exponential variables with mean $1/r$, which model the lifetimes. We build $\mathbf{Z}^x$ and $\mathbf{Z}^y$ by induction. First, for all $t \in [0, d_\emptyset), \ X^\emptyset_t= x + t$ and $Y^\emptyset_t= y + t$. We set $\alpha(\emptyset) =0$. Then, for all $u \in \mathcal{T}$ and $k \in \{ 1, 2  \}$, we set $\alpha(uk) = \alpha(u) + d_u$ and
$$
\forall t \in [\alpha(uk), \alpha(uk) + d_{uk}), \ X^{uk}_t=  \frac{1}{2} X^{u}_{\alpha(uk)-} + t - \alpha(uk)
$$ 
and $Y^{uk}_t =  Y^{u}_{\alpha(uk)-} / 2 + t - \alpha(uk)$. Finally we have $\mathcal{V}_t = \{ u \in \mathcal{T} \ | \  \alpha(u) \leq t < \alpha(u) + d_u  \}$ and
$$
\mathbf{Z}^x_t = \sum_{ u \in \mathcal{V}_t} \delta_{X^u_t} \ \text{ and } \ \mathbf{Z}^y_t = \sum_{ u \in \mathcal{V}_t} \delta_{Y^u_t}.
$$
We observe that, for any cell $u$, the trajectories of $X^u$ and $Y^u$ are parallel (because they are linear). When a branching occurs, $\sum_{u \in \mathcal{V}_t} |X^u_t - Y^u_t|$ is constant. Hence, we easily deduce that
$$
\sum_{u \in \mathcal{V}_t} |X^u_t - Y^u_t| = |x-y|.
$$
Finally we have, for all $t\geq0$,
\begin{align*}
 W_{|\cdot|} (\mathbf{Z}^{x}_t ,\mathbf{Z}^{y}_t)
&\leq \sum_{u \in \mathcal{V}_t} |X^u_t - Y^u_t| \\
&\leq |x-y|.
\end{align*}
Dividing by $\E[N_t]=e^{-rt}$, we obtain the first bound. For the second bound, a similar computation gives
$$
W_{W_{|\cdot|}} \left(\mathcal{L}\left(\frac{\mathbf{Z}^{x}_t}{N_t}\right) , \mathcal{L}\left(\frac{\mathbf{Z}^{y}_t}{N_t}\right)\right) \leq \E\left[\frac{1}{N_t}\right] |x-y|.
$$
The process $(N_t)_{t\geq0}$ is know to be the Yule's process. It is geometrically distributed with parameter $e^{-rt}$, so we have
$$
\E\left[\frac{1}{N_t}\right]= \frac{rt e^{-rt}}{1- e^{-rt}}.
$$
It ends the proof.
\end{proof} 

\begin{Rq}[Generalisation of Theorem \ref{vitess}]
In the proof of	Theorem \ref{vitess}, we only need that, for all $n\in \N^*$, $\theta\in[0,1]$,$t\geq0$, and $x,y\in E$
$$
 \sum_{j=1}^{n}  | F^{(k)}_j ( X_{t}, \theta ) - F^{(k)}_j( Y_{t}, \theta )| \leq |x-y|
$$
where $X,Y$ are generated by $G$ and start respectively from $x,y$. For instance, we can consider that $X$ is a continuous l\'evy process and the division is a sub-critical fragmentation; namely 
$$
\forall x\in E, \ \forall k \in \N^*, \forall j \leq k,  \ F^{(k)}_j (x, \Theta)= \Theta^{k}_j x,
$$
where $(\Theta^{k}_j)_{j,k}$ is a family of random variable verifying
$$
 \ \sum_{j=1}^k \Theta_{j}^{k} \leq 1 \ \mbox{ and } \ \forall j \in \{1,\dots,k\}, \  \ \Theta_{j}^{k} \in [0,1]. 
$$
\end{Rq}

Even if we do not find an explicit bound, we are able to prove a Wasserstein convergence.

\begin{lem}[Wasserstein convergence]
Under the assumptions of Theorem \ref{vitess}, we have
$$
\lim_{t \rightarrow + \infty} W_{|\cdot|} \left(\frac{\mathbf{Z}_t}{N_t},  \pi \right)= 0 \text{ in probability}.
$$
\end{lem}
\begin{proof}
As $x \mapsto 1+x$ is a Lyapounov function for the auxiliary process, we have
$$
\lim_{t \rightarrow + \infty} \frac{\mathbf{Z}_t}{N_t}(f) = \pi (f) \text{ in probability},
$$
for all function $f$ such that $|f(x)| \leq C(1 + x)$. The convergence also holds in distribution. By the Skorohod's Theorem, in another probability space, we have,
$$
\lim_{t \rightarrow + \infty} \frac{\mathbf{Z}_t}{N_t}(f) = \pi (f) \text{ a.s.}
$$
for all continuous bounded function and for $f(x)=x$. This convergence is equivalent to the Wasserstein convergence. Thus, by a classical argument of discreteness (Varadarajan Theorem type), we get,
$$
\lim_{t \rightarrow + \infty} W_{|\cdot|} \left(\frac{\mathbf{Z}_t}{N_t},  \pi \right)= 0 \text{ a.s.}.
$$
Hence, in our probability space we get that $\lim_{t \rightarrow + \infty} W_{|\cdot|} \left(\mathbf{Z}_t/N_t,  \pi \right)= 0 $ in distribution. And as the convergence is deterministic, we get the result.
\end{proof}

\subsection{Asymmetric mitosis : Macroscopic approximation}
Now, we do not assume that the division is symmetric. We assume that $F^{(2)}_1 (x,\theta)=F^{-1}(\theta) x$ and $F^{(2)}_2 (x,\theta)= (1-F^{-1}(\theta)) x$. We recall that $F(x) = 1-F(1-x)$. In this case, Equation \eqref{eq:thintro} becomes
$$
\partial_t n(t,x)  + \partial_x n(t,x) + r (x) \ n(t,x) =  2  \E [ \frac{1}{\Theta} r (x/\Theta) n(t,x/\Theta)],
$$
where $n(t,.)$ is the density of $\mathbf{X}_t$. In particular, we deduce that the following P.D.E. has a weak solution:
$$
\partial_t n(t,x)  + \partial_x n(t,x) + r (x) n(t,x) =  \int_x^{+ \infty}  b(x,y) n(t,y) dy 
$$
where $b$ verify the following properties:
\begin{align}
\label{b1}
& b(x,y) \geq 0, b(x,y)= 0 \ \mbox{ for } y<x \\
& \int_0^{+\infty} b(x,y)dx = 2 r(y) \\
& \int_0^{+\infty} x b(x,y)dx = y r(y) \\
& \label{b2} b(x,y)=b(y-x,y).
\end{align}

This equation was studied in \cite{Pert1}. Here, 
\begin{equation}
\label{eq:b}
b(x,y)=\frac{2}{y} r(y) g(\frac{x}{y}),
\end{equation}
where $g$ is the weak density of $F$. We easily prove the equivalence between to verify \eqref{eq:b} and (\ref{b1} - \ref{b2}). Our aim in this section is to describe the limit of the fluctuation process. It is defined  by:
$$
\forall t \in [0, T], \forall n \in \N^*, \ \eta^{(n)}_t = \sqrt{n} (\mathbf{X}^{(n)}_t - \mathbf{X}_t).
$$

\begin{theo}[Central limit Theorem for asymmetric size-structured population]
\label{TCL}
Let $T>0$. Assume  that $\eta^{(n)}_0$ converges in distribution and that
\begin{equation}
\label{eq:X0fini}
\E\left[\sup_{n \geq 1} \int_E (1 + x^2) \ \mathbf{X}^{(n)}_0 (dx) \right] < + \infty.
\end{equation}
Then the sequence $(\eta^{(n)})_{n \geq 1}$ converges in $\mathbb{D}([0,T], C^{-2,0})$ to the unique solution of the evolution equation: for all $f \in C^{2,0}$,
\begin{align}
\label{eqevo}
\eta_t(f) 
&= \eta_0(f) \\
&+ \int_0^t \int_0^{+\infty} f'(x)  + r (x) \left( \int_0^1 f(q x) + f((1-q)x) F(dq)-f(x)\right) \eta_s(dx)  ds \nonumber  \\
&+ \widetilde{\mathbf{M}}_t(f), \nonumber
\end{align}
where $\widetilde{\mathbf{M}}(f)$ is a martingale and a Gaussian process with bracket:
$$
\langle \widetilde{\mathbf{M}}(f) \rangle_t = \int_0^t \int_0^{+\infty} 2 f'(x) f(x) + 2 r(x) \int_0^1 ( f(qx) -f(x))^2 F(dq) \ \mathbf{X}_s (dx) ds. 
$$
And $C^{2,0}$ is the set of $C^2$ functions, such that $f,f',f''$ vanish to zero when $x$ tends to infinity. $C^{-2,0}$ is its dual space.
\end{theo}

Lemma \ref{semimart2} gives
$$
\forall \ t \geq 0, \ \eta^{(n)}_t = \eta^{(n)}_0 + \widetilde{\mathbf{V}}^{(n)}_t + \widetilde{\mathbf{M}}^{(n)}_t,
$$
where for any $f$ smooth enough,
$$
\widetilde{\mathbf{V}}^{(n)}_t  (f) = \int_0^t \int_0^{+\infty} f'(x)  + r(x) \left( \int_0^1 f(q x) + f((1-q)x) F(dq)-f(x)\right) \ \eta^{(n)}_s(dx) ds,
$$
and $\widetilde{\mathbf{M}}^{(n)}$ is a martingale with bracket:
\begin{equation}
\label{croch}
\langle \widetilde{\mathbf{M}}^{(n)} (f) \rangle_t = \int_0^t \int_0^{+\infty} 2 r(x) \int_0^1 ( f(qx) -f(x))^2 F(dq) \ \mathbf{X}^{(n)}_s (dx)  \  ds .
\end{equation}

As the set of signed measure is not metrizable, we can not adapt the proof of Theorem \ref{th:Grandpopintro}. Following \cite{M98,TheseT}, we consider $\eta^{(n)}$ as an operator in a Sobolev space, and use the Hilbertian properties of this space to prove tightness. See for instance \cite{M84} for condition to prove tightness on Hilbert spaces. Let us explain the Sobolev space that we will use. Let $p > 0$ and $j\in \N$. The set $W^{j,p}$ is the closure of $C^\infty_c$, which is the set of functions of class $C^{\infty}$ from $\R_+$ into $\R$ with compact support, embedded with the following norm:                                                                                                                                                  
$$
\forall f \in W^{j,p}, \ \Vert f \Vert_{W^{j,p}}^2 = \sum_{k =0}^j \int_0^\infty \left( \frac{f^{(k)}(x)}{1+x^p} \right)^2 dx.
$$ 
The set $W^{j,p}$ is an Hilbert space and we denote by $W^{-j,p}$ its dual space. Let $C^{j,p}$ be the space of function $f$ of class $C^j$ such that:
$$
\forall k \leq j, \ \lim_{x \rightarrow + \infty} \frac{f^{(k)}(x)}{1+x^p} = 0.
$$ 
We embed it with the following norm:
$$
\forall f \in C^{j,p}, \ \Vert f \Vert_{C^{j,p}} = \sum_{k =0}^j \sup_{x \geq 0} \frac{f^{(k)}(x)}{1+x^p}.
$$
The set $C^{j,p}$ is also a Banach space and we denote by $C^{-j,p}$ its dual space. These spaces verify the following continuous injection \cite{M98,A75}:
\begin{equation}
\label{injsob}
 C^{j,p} \subset W^{j,p+1} \ \text{ and } \  W^{1+j,p} \subset C^{j,p}.
\end{equation}
Or equivalently, if $f$ is smooth enough,
$$
\Vert f \Vert_{W^{j,p+1}} \leq C \Vert f \Vert_{C^{j,p}} \ \text{ and } \   \Vert f \Vert_{C^{j,p}} \leq C \Vert f \Vert_{W^{j+1,p}}.
$$
The first embedding/inequality prove that the tightness in $W^{j,p+1}$ implies the tighness in $C^{j,p}$. The second is useful for some upper bounds. For instance, we have
\begin{lem}
If $(e_k)_{k\geq 1}$ is a basis of $W^{2,1}$ then we have, for all $k\geq 0$ and $x\in E$,
$$
 \sum_{k \geq 1 } e_k(x)^2 \leq C (1+x^2).
$$
\end{lem}
\begin{proof}
$ \delta_x:f \mapsto f(x)$ is an operator on $W^{2,1}$. We have, for all $f \in W^{2,1}$,
$$
 |\delta_x f| \leq (1+x) \Vert f \Vert_{C^{0,1}} \leq C (1+x) \Vert f \Vert_{W^{1,1}} \leq C (1+x) \Vert f \Vert_{W^{2,1}}   
$$
But, by Parseval's identity we get,
$$
\Vert \delta_x \Vert_{W^{-2,1}}^2 = \sum_{k\geq 1} e_k(x)^2,
$$
which completes the proof.
\end{proof}
We introduce the trace $\left( \langle \langle \widetilde{\mathbf{M}}^{(n)} \rangle \rangle_t \right)_{t \geq 0}  $ of $\left(\widetilde{\mathbf{M}}^{(n)}_t\right)_{t \geq 0}$. It is defined such that $$\left( \Vert \widetilde{\mathbf{M}}^{(n)}_t \Vert_{W^{-2,1}}^2 - \langle \langle \widetilde{\mathbf{M}}^{(n)} \rangle \rangle_t \right)_{t\geq0}$$ is a local martingale. Then since
$$
\left\Vert \widetilde{\mathbf{M}}^{(n)}_t \right\Vert_{W^{-2,1}}^2 = \sum_{k \geq 1} \widetilde{\mathbf{M}}^{(n)}_t(e_k),
$$
where $(e_k)_{k \geq 1}$ is a basis of $W^{2,1}$. Then by (\ref{croch}), we get
$$
\langle \langle \widetilde{\mathbf{M}}^{(n)} \rangle \rangle_t = \sum_{k \geq 1} \int_0^t \int_0^{+\infty} 2 r(x) \int_0^1 ( e_k(qx) - e_k(x))^2 F(dq) \mathbf{X}^{(n)}_s (dx) ds.
$$
Now, we first prove the tightness of $(\eta^{(n)})_{n\geq 1}$ then Theorem \ref{TCL}
\begin{lem}
\label{etatight}
$(\eta^n)_{n \geq 1}$ is tight in $\mathbb{D}([0,T],W^{-2,1})$
\end{lem}
\begin{proof}
By \cite[Theorem 2.2.2]{critere} and \cite[Theorem 2.3.2]{critere} (see also \cite[Lemma C]{M98}), it is enough to prove 
\begin{enumerate}
\item $\E\left[\sup_{s \leq t} \Vert \eta^n_s \Vert^2_{W^{-2,1}}\right] < + \infty $,
\item $\forall n \in \N, \ \forall \varepsilon, \rho >0, \ \exists \delta>0$ such that for each stopping times $S_n$ bounded by $T$
$$
\limsup_{n \rightarrow + \infty} \sup_{0 \leq u \leq \delta} \p \left( \left\Vert \widetilde{\mathbf{V}}^{(n)}_{S_n + u} - \widetilde{\mathbf{V}}_{S_n} \right\Vert_{W^{-2,1}} \geq \eta \right) \leq \varepsilon,
$$
$$
\limsup_{n \rightarrow + \infty} \sup_{0 \leq u \leq \delta} \p \left(\left| \left< \left< \widetilde{\mathbf{M}}^{(n)} \right> \right>_{S_n + u} - \left< \left< \widetilde{\mathbf{M}}^{(n)} \right> \right>_{S_n} \right| \geq \eta \right) \leq \varepsilon.
$$
\end{enumerate}
For the first point, using lemma \ref{CS}, there exists $C>0$ such that
\begin{align*}
\sum_{k\geq 1} \left\langle \widetilde{\mathbf{M}}^{(n)}_t (e_k) \right\rangle 
&\leq \int_0^t 2 \bar{r} \int_0^1 2 \sum_{k \geq 1} e_k^2(qx) + 2 \sum_{ k \geq 1} e_k^2(x) \ F(dq) \mathbf{X}^{(n)}_s (dx) ds  \\
&\leq C \mathbf{X}^{(n)}_0 (1 + x).\\
\end{align*}
Then, since
$$
\left\Vert \widetilde{\mathbf{M}}^{(n)}_t \right\Vert^2_{W^{-2,1}} = \sum_{k\geq 1} \left(\widetilde{\mathbf{M}}^{(n)}_t (e_k)\right)^2,
$$
Doob's inequality and \eqref{eq:X0fini} gives
$$
\E\left[\sup_{t \in [0,t]} \left\Vert \widetilde{\mathbf{M}}^{(n)}_t \right\Vert_{W^{-2,1}}^2 \right] \leq C',
$$
where $C'>0$. Then there exits $C''>0$ such that
$$
\left\Vert \eta^{(n)}_t \right\Vert^2_{W^{-2,1}} \leq \left\Vert \eta^{(n)}_0 \right\Vert^2_{W^{-2,1}} + \left\Vert \widetilde{\mathbf{V}}^{(n)}_t \right\Vert^2_{W^{-2,1}} +\left\Vert \widetilde{\mathbf{M}}^{(n)}_t \right\Vert^2_{W^{-2,1}} \leq C'' + \left\Vert \widetilde{\mathbf{V}}^{(n)}_t \right\Vert^2_{W^{-2,1}}.
$$
And as
$$
\left\Vert \widetilde{\mathbf{V}}^{(n)}_t \right\Vert^2_{W^{-2,1}}  \leq C \int_0^t \sup_{w\leq s} \left\Vert \eta^{(n)}_s \right\Vert^2_{W^{-2,1}} ds,
$$
the Gronwall Lemma gives
$$
\E\left[\sup_{s \leq t} \left\Vert \eta^{(n)}_s \right\Vert^2_{W^{-2,1}}\right] \leq K,
$$
for a certain constant $K$. Finally for the second point, we have
\begin{align*}
\E\left[\left\Vert \widetilde{\mathbf{V}}^{(n)}_{S_n + u} - \widetilde{\mathbf{V}}^{(n)}_{S_n} \right\Vert_{W^{-2,1}} \right] 
&\leq \E\left[ K' \int_{S_n}^{S_n + u} \sup_{s \leq T} \left\Vert \eta^{(n)}_s \right\Vert^2_{W^{-2,1}} \right]\\
& \leq K'' u.
\end{align*}
Here $K',K''$ are two constants. Using the Markov-Chebyshev inequality, we prove the Aldous condition. We similarly prove that $\langle \langle \widetilde{\mathbf{M}}^{(n)} \rangle \rangle$ verifies the Aldous condition. We deduce that $(\eta^{(n)})_{n \geq 1}$ is tight.
\end{proof}

\begin{proof}[Proof of Theorem \ref{TCL}]
Let $\widetilde{\mathbf{M}}$ be a continuous Gaussian process with quadratic variation verifying, for every $f \in C^{2,0}$ ($\subset W^{2,1}$) and $t \in [0, T]$,
$$
\langle \widetilde{\mathbf{M}} (f) \rangle_t =\sum_{k \geq 1} \int_0^t \int_0^{+\infty} 2 r(x) \int_0^1 ( f(qx) - f(x))^2 F(dq) \mathbf{X}_s (dx).
$$
Since there exists $C_f$ such that
$$
\forall f \in C^{2,0}, \ \sup_{t \in [0,T]} |\widetilde{\mathbf{M}}^{(n)}(f)| \leq \frac{C_f}{\sqrt{n}} ,
$$
and $\langle \widetilde{\mathbf{M}}^{(n)} \rangle_t$ converges in law to $\langle \widetilde{\mathbf{M}} \rangle_t$, then by \cite[Theorem 3.11 p.473]{Jac}, $\widetilde{\mathbf{M}}^{(n)}(f)$ converges to $\widetilde{\mathbf{M}}(f)$ in distribution, as $n$ tends to $\infty$.\\

By Lemma \ref{etatight} and (\ref{injsob}) , the sequence $(\eta^{(n)})_{n \geq 1}$ is also tight in $C^{-2,0}$. Let $\eta$ be an accumulation point. Since its martingale part $\widetilde{\mathbf{M}}$ in its Doob's decomposition is almost surely continuous, then $\eta$ is also almost surely continuous. Hence, $\eta$ is a solution of (\ref{eqevo}). Using Gronwall's inequality, we obtain the uniqueness of this equation, in $C([0,T],C^{-2,0})$, up to a Gaussian white noise $\widetilde{\mathbf{M}}$. We deduce the announced result.
\end{proof}

\section{Another two examples}

\label{sect:autres-exemples}

\subsection{Space-structured population model}

Here, we study an example which can models the cells localisation. One cell moves following a diffusion on $E \subset \R^d$, $d\geq1$, and when it dies, its offspring is localised at the same place. Hence, in all this section the branching is local; that is
$$
\forall k \geq 0, \forall j \leq k, \forall x \in E, \forall \theta \in [0,1], \  F_j^{(k)}(x, \theta) = x.
$$

\subsubsection{Branching Ornstein Uhlenbeck}
In this subsection, we consider the model of \cite[Example 10]{EHK10}. Assume that $G$ is given by
$$
Gf(x)=\frac{1}{2} \sigma^2 \Delta f(x) - g x . \nabla f(x),
$$
where $\sigma,g>0$, f is smooth, $x\in \R^d$. Also assume that the division is dyadic, that is $p_2=1$, with rate
$$
r(x)=b x^2 +a,
$$
where $a,b\geq 0$ and $a$ or $b$ is not null. Here $x^2= \Vert x\Vert^2 =x.x$. If $g > \sqrt{2b}$ then we add the following notations:
$$
\Gamma = \frac{ g -\sqrt{g^2- 2b\sigma^2}}{2 \sigma^2} \ \text{ and } \ \alpha= \sqrt{g^2 - 2b\sigma^2}.
$$
We also denote by $\pi_\infty$ the Gaussian measure whose density is defined by
$$
x\mapsto \left(\frac{\alpha}{\pi \sigma^2}\right) \exp\left(- \frac{\alpha}{\sigma^2} x^2 \right).
$$
From our main theorem, we deduce

\begin{coro}[Limit theorem for an branching Ornstein-Uhlenbeck process]
If $g > \sigma \sqrt{2b}$ and $X_0^{\emptyset} =x \in \R^d$ then, for any continuous and bounded $f$, we have
$$
\lim_{t \rightarrow + \infty} \frac{1}{N_t}  \sum_{u \in \mathcal{V}_t} f(X^u_t) = \frac{\int_{\R^d} f(y) e^{\Gamma y^2} \pi_\infty(dy)}{\int_{\R^d} e^{\Gamma y^2} \pi_\infty(dy)},
$$
in probability. In particular,
$$
\E[N_t]= e^{\lambda t + \Gamma x^2}  \left(\frac{\alpha}{\pi \sigma^2}\right) \int_{\R^d} e^{-\Gamma y^2} \exp\left(- \frac{\alpha (y -x e^{- \alpha t/\sigma^2})^2}{\sigma^2 (1- e^{- 2\alpha t/\sigma^2} )} \right) dy,
$$
where $\lambda=\frac{g -\sqrt{g^2- 2b\sigma^2}}{2 } + a$ is the Malthus parameter.
\end{coro}
\begin{proof}
If $V :x \mapsto e^{\lambda x^2}$ then it is an eigenvector of $\mathcal{G}$, which is defined for any smooth $f$ by
$$
\mathcal{G} f (x) = Gf(x) + r(x) f(x).
$$ 
We conclude the proof using Theorem \ref{thintro} and Lemma \ref{lem:WMTO}.
\end{proof}

\begin{Rq}[Another eigenelement]
Note that if $V_2 :x \mapsto e^{\lambda_2 x^2}$ then it is an eigenvector of $\mathcal{G}$, associated to the eigenvalue
$$
\lambda_2 = \frac{g +\sqrt{g^2- 2b\sigma^2}}{2 } + a.
$$
But in this case, the auxiliary process is not ergodic and we are not able to deduce any convergence from our main theorem.
\end{Rq}

\subsubsection{General case} Let us assume that $G$ is the generator of a diffusive Markov process. If the state space $E$ is bounded then we can find sufficient conditions to the eigenproblem in \cite[section 3]{Pinsky} and \cite[Theorem 5.5]{Pinsky}. For instance, under some assumptions, we have
$$
\lambda_0 = \lim_{t \rightarrow + \infty} \ln\left( \sup_{x\in E} \E\left[N_t \ | \ X_0^\emptyset = x \right] \right).
$$

If $E$ is not bounded  then we can see \cite{HS96,RS78}. This example is developed in \cite{EHK10}. They prove a strong law of large number, which is close to Theorem \ref{thintro}.

\subsection{Self-similar fragmentation} 
Self-similar mass fragmentation processes are characterised by
\begin{itemize}
\item the index of self-similarity $\alpha\in \R$;
\item a so-called dislocation measure $\nu$ on $\mathcal{S}=\{ s=(s_i)_{i\in \N} \ | \ \lim_{i \rightarrow + \infty} s_i = 0 , 1 \geq s_j \geq s_i \geq 0, \forall j \leq i \}$ which satisfies
$$
\nu ({ 1,0,0,..}) =0 \ \text{ and } \ \int_{\mathcal{S}} (1-s) \nu(ds) < + \infty.
$$
\end{itemize}
If $\nu(\mathcal{S})< + \infty$ then the dynamics is as follows:
\begin{itemize}
\item a block of mass $x$ remains unchanged for exponential periods of time with parameter $x^\alpha \nu(\mathcal{S})$;
\item a block of mass $x$ dislocates into a mass partition $xs$, where $s\in \mathcal{S}$, at rate $\nu(ds)$;
\item there are finitely many dislocations over any finite time horizon.
\end{itemize}
The last point is not verified when $\nu(\mathcal{S}) = + \infty$. In this case, there is a countably infinite number of dislocations over any finite time horizon. So, when $\nu(\mathcal{S})<+\infty$, our setting capture this model with the following parameters: 
$$
G = 0, \ r(x) = x^\alpha \nu(\mathcal{S}), \ \int_0^1 \sum_{k\geq0} p_k(x) \sum_{j=1}^k f(F_j^{k}x,\theta) d\theta = \int_\mathcal{S} \sum_{i\geq0} f(s_i x) \frac{\nu(ds)}{ \nu(\mathcal{S})}.
$$
Hence, in this case we have
$$
\mathcal{G} f(x) = x^\alpha \nu(\mathcal{S}) \sum_{i \geq 0 } \left( \int_\mathcal{S} \sum_{i \geq 0} f(s_i x) \frac{\nu(ds)}{\nu(\mathcal{S})} - f(x) \right),
$$
and $V:x\mapsto x^p$ is an eigenevector. See \cite{B06} for further details.\\
\\
\textbf{Acknowledgment:} I would like to express my gratitude to my Ph.D. supervisor Djalil Chafa\"{i} for his encouragement, and essential help on the form and the content of this paper. I also thank, Viet Chi Tran for pointing out some references and for fruitful discussions.

\bibliographystyle{abbrv} %alpha abbrv
\bibliography{ref1}

\end{document}